\documentclass[11pt]{article}
\usepackage{hyperref}
\usepackage{cite,amsmath,amsfonts,amssymb,amsthm,graphicx,verbatim}
\usepackage{fullpage}

\newtheorem{thm}{Theorem}[section]
\newtheorem{lem}[thm]{Lemma}
\newtheorem{cor}[thm]{Corollary}
\newtheorem{dfn}[thm]{Definition}
\newtheorem{conj}{Conjecture}

\newtheorem{question}{Question}
\newtheorem{claim}{Claim}
\renewcommand{\qed}{\hfill{\rule{2mm}{2mm}}}
\renewenvironment{proof}[1][]{\begin{trivlist}
\item[\hspace{\labelsep}{\bf\noindent Proof\ #1:\/}] }{\qed\end{trivlist}}

\providecommand{\cF}{\mathcal{F}}
\providecommand{\cG}{\mathcal{G}}
\providecommand{\cT}{\mathcal{T}}

\providecommand{\inner}[2]{\langle #1,#2 \rangle}
\DeclareMathOperator*{\Exp}{\mathbb{E}}

\DeclareMathOperator{\rank}{rank}
\DeclareMathOperator{\isomorphicto}{\cong} 
\DeclareMathOperator{\eqdef}{=:}
\DeclareMathOperator{\Span}{Span}

\newcommand{\beq}[1]{\begin{equation}\label{#1}}
\newcommand{\enq}[0]{\end{equation}}
\def\elong{\mathrel\rhd\joinrel-\joinrel\lhd}

\newcommand{\C}[2]{{\binom{#1}{{#2}}}}

\newcommand{\eps}[0]{\varepsilon}
\renewcommand{\epsilon}[0]{\varepsilon}

\newcommand\ip[1]{{\langle {#1} \rangle}}

\newcommand{\ra}[0]{\rightarrow}

\newcommand{\B}[0]{{\cal B}}

\newcommand{\F}[0]{{\cal F}}
\newcommand{\G}[0]{{\cal G}}

\newcommand{\T}[0]{{\cal T}}

\newcommand{\barB}[0]{{\overline{B}}}

\def\deg{\mathrm{deg}}

\newcommand{\dnote}[1]{{\bf (David:} {#1}{\bf ) }}
\newcommand{\enote}[1]{}
\newcommand{\ynote}[1]{}

\newcommand{\remove}[1]{}

\def\trium{\triangle\mbox{umvirate}}

\begin{document}
 \title{Triangle-intersecting Families of Graphs}
 \author {
David Ellis\thanks{St John's College, Cambridge, United Kingdom.}
\and
Yuval Filmus
\thanks{Department of Computer Science,
University of Toronto. email: yuvalf@cs.toronto.edu.
Supported by NSERC.}
\and
Ehud Friedgut
  \thanks{Institute of Mathematics, Hebrew
    University, Jerusalem, Israel. email: ehud.friedgut@gmail.com
    Research supported in part by the Israel Science Foundation, grant
    no. 0397684. }
\\ {\em Dedicated to Vera T. S\'{o}s on occasion of her 80th birthday}
}
 \date{September 2010}

{
\renewcommand{\thefootnote}{\fnsymbol{footnote}}
\footnotetext{MSC 2010 subject classifications:  05C35,05D99}
\footnotetext{Key words and phrases:  Intersecting families,
 Graphs, Discrete Fourier analysis.}
\maketitle
}

\begin{abstract}
A family of graphs $\cF$ is {\em triangle-intersecting} if for every $G,H\in\cF$, $G \cap H$ contains a triangle.
A conjecture of Simonovits and S\'{o}s from 1976 states that the largest triangle-intersecting
families of graphs on a fixed set of $n$ vertices are those obtained by fixing a specific triangle and taking all
graphs containing it, resulting in a family of size $\frac{1}{8}2^{\binom{n}{2}}$. We prove this conjecture and
some generalizations (for example, we prove that the same is true of odd-cycle-intersecting families, and we obtain best possible bounds on the size of the family under different, not necessarily uniform, measures).
We also obtain stability results, showing that almost-largest triangle-intersecting families have approximately
the same structure.

\end{abstract}

\section{Introduction}\label{introduction}
A basic theme in the field of extremal combinatorics
is the study of the largest size of a structure (e.g. a family of sets) given some combinatorial information
concerning it (e.g. restrictions on the intersection of every two sets in the family.) The fundamental example of
this is the Erd\H{o}s-Ko-Rado theorem \cite{EKR} which bounds the size of an {\em intersecting family}
of $k$-element subsets of an $n$-element set (meaning a family in which any two sets have non-empty intersection). For $k < n/2$, the simple answer is that the unique
largest intersecting families are those obtained by fixing an element and choosing all $k$-sets containing it.
This theorem is amenable to countless directions of generalizations: demanding larger intersection size, having
some arithmetic property of the intersection sizes, removing the restriction on the size of the sets while introducing
some measure on the Boolean algebra of subsets of $\{1,\ldots,n\}$ etc. etc. Usually, the aesthetically pleasing
theorems are those, like the EKR theorem, where the structure of the extremal families is simple to describe,
often by focussing on a small set of elements through which membership in the family is determined.

A beautiful direction suggested by Simonovits and S\'{o}s is that of introducing structure on the ground set,
namely considering subgraphs of the complete graph on $n$ vertices. They initiated the investigation
in this direction with the following definition and question.
\begin{dfn}
A family of graphs $\cF$ is {\em triangle-intersecting} if for every $G,H \in \cF$, $G \cap H$ contains a triangle.
\end{dfn}
\begin{question}[Simonovits-S\'{o}s]
What is the maximum size of a triangle-intersecting family of subgraphs of the complete graph on $n$ vertices?
\end{question}
They raised the natural conjecture that the largest families are precisely those given by fixing a triangle and taking all graphs containing this triangle. In this paper we prove their conjecture.
\begin{thm}\label{main-}
Let $\cF$ be a triangle-intersecting family of graphs on $n$ vertices. Then $|\cF| \leq \frac 1 8 2^{\binom {n}{2}}$.
Equality holds if and only if \(\mathcal{F}\) consists of all graphs containing a fixed triangle.
\end{thm}

Our main result in this paper is actually a strengthening of the above in several aspects. First,
we relax the condition that the intersection of every two graphs in the family contains a triangle,
 and demand only that it contain an odd cycle (i.e. be non-bipartite). Secondly, we allow the size of the family to be measured
not only by the uniform measure on the set of all subgraphs of $K_n$, but rather according to the product measure
of random graphs, $G(n,p)$, for any $p \leq 1/2$. Thirdly, for the case of the uniform measure, we relax the condition
that for every two graphs $G$ and $H$ in the family, $G \cap H$ contains a triangle, to the condition that
$G$ and $H$ `agree' on some triangle --- i.e. that there exists a triangle that is disjoint from the
symmetric difference of $G$ and $H$.  Furthermore, we prove a stability result: any triangle-intersecting family
that is sufficiently close in measure to the largest possible measure is actually close to a {\em bona-fide} extremal family.
Finally, we observe that our proofs can be pushed further without much effort to prove a similar result about
(not necessarily uniform) hypergraphs --- a result one might refer to as dealing with Schur-triple-intersecting
families of binary vectors.

Before making all of the above precise and expanding a bit on our methods, let us introduce some necessary notation and definitions
and review some relevant previous
work.
\subsection{Notation and main theorems}
Let $n$ be a positive integer, fixed throughout the paper. The power set of $X$ will be denoted $\mathcal{P}(X)$. As usual, $[n]$ denotes the set $\{1,2,\ldots,n\}$.
Also, $[n]^{(k)}$ will denote $\{S \subseteq [n] : |S|=k\}$.
It will be convenient to think of the set of all subgraphs
of $K_n$ as the Abelian group $\mathbb{Z}_2^{[n]^{(2)}}$ where the group operation, which we denote by $\oplus$,
is the symmetric difference (i.e. $H \oplus G$ is the graph whose edge set is the
symmetric difference between the edge sets of $G$ and $H$); we will also use the notation $\Delta$ for the same operator. We will write $\overline{G}$ for the complement of a graph $G$.
Since we identify graphs with their edge sets, we will write $|G|$ for the number of edges in $G$, and $v(G)$ for the
number of non-isolated vertices in $G$. We will denote the fact that $G$ and $H$ are isomorphic by $G \isomorphicto H$. If $G$ is the \emph{disjoint} union of two graphs $G_1,G_2$ (that is, $G_1,G_2$ have no edges in common), then we will write $G = G_1 \sqcup G_2$.
\begin{dfn}
A family $\cF$ of subgraphs of $K_n$ is {\em triangle-intersecting} (respectively {\em odd-cycle-intersecting})
if for every $G,H \in \cF$, $G \cap H$ contains a triangle (respectively an odd cycle).
We will say that $\cF$ is {\em triangle-agreeing} (respectively {\em odd-cycle-agreeing})
if for every $G,H \in \cF$, $\overline{G \oplus H} $ contains a triangle (respectively an odd cycle).
\end{dfn}
Note that $G \cap H$ is contained in $\overline{G \oplus H}$, so a triangle-intersecting family is also
triangle-agreeing.

Given $\cF$, a family of subgraphs of $K_n$, we will want to measure its size according
to skew product measures: for any $p \in [0,1]$ and graph $G$ on $n$ vertices
we will denote by $\mu_p(G)$ the probability that $G(n,p)=G$, i.e.
$$\mu_{p}(G) = p^{|G|} (1-p)^{\binom{n}{2}-|G|},$$
and for a family of graphs $\cF$ we define \(\mu_{p}(\mathcal{F})\) to be the probability that \(G(n,p) \in \mathcal{F}\), i.e.
$$\mu_p(\cF) = \sum_{G \in \cF} \mu_p(G).$$
When $p$ is fixed (e.g. throughout the section where $p=1/2$) we will drop the subscript and simply write $\mu(G)$ and  $\mu(\cF)$.
For any two functions $f, g \colon \mathbb{Z}_2^{[n]^{(2)}} \ra \mathbb{R}$ we define their inner product as
\[ \inner{f}{g} = \Exp(f \cdot g) = \sum_{G} \mu(G) f(G) g(G). \]
We will denote the graph on $n$ vertices with no edges by $\emptyset$.  A $k$-forest is any forest with $k$ edges.
The graph on four vertices with 5 edges will be denoted by $K_4^-$. A {\em biconnected component} of a graph \(G\) means a maximal biconnected subgraph of \(G\) (i.e. it need not be an entire component).

If \(X\) is a finite set, \(\mathcal{P}(X)\) will denote the power set of \(X\), the set of all subsets of \(X\). Identifying a set with its characteristic function, we will often identify \(\mathcal{P}(X)\) with \(\{0,1\}^{X} = \mathbb{Z}_{2}^{X}\). A family \(\mathcal{F}\) of subsets of \(X\) is said to be an {\em up-set} if whenever \(S \in \mathcal{F}\) and \(T \supset S\), we have \(T \in \mathcal{F}\). The notation ${\bf 1}_P$ for a predicate $P$ means $1$ if $P$ holds, and $0$ if $P$ doesn't hold. If \(A\) is an Abelian group, and \(Y \subset A\), we write \(\Gamma(A,Y)\) for the Cayley graph on \(A\) with generating set \(Y\), meaning the graph with vertex-set \(A\) and edge-set \(\{\{a,a+y\}:\ a \in A,\ y \in Y\}\).

A {\em Triangle junta} is a family of all subgraphs of $K_n$ with a prescribed intersection with a given triangle.
In the special case of the triangle junta being the family of all graphs containing a given triangle, we will
call this family a $\trium$. (Don't ask us how this is pronounced.)

Our main theorem is the following.
\begin{thm}\label{main}
\begin{itemize}
\item{{\bf [Extremal families]}}
Let \(p \leq 1/2\), and let $\cF$ be an odd-cycle-intersecting family of subgraphs of $K_n$. Then $\mu_p(\cF) \le p^3$, with equality if and only if $\cF$ is a $\trium$. Furthermore, in the case $p = 1/2$, if $\cF$ is odd-cycle-agreeing then $\mu(\cF) \le 1/8$, with equality if and only if $\cF$ is a triangle
junta.
\ynote{The constant $c$ used to depend on $p$.}
\item{{\bf [Stability]}} For each $p \leq 1/2$ there exists a constant $c_p$ (bounded for $p \in [\delta,1/2]$, for any fixed \(\delta >0\)) such that for any $\epsilon \ge 0$,
if $\cF$ is an odd-cycle-intersecting
family with $\mu_{p}(\mathcal{F}) \geq p^3 - \epsilon$ then there exists a $\trium$ $\T$ such that
$$
\mu_p ( \T \Delta \cF) \le c_p \epsilon.
$$
For $p=1/2$, the corresponding statement holds for odd-cycle-agreeing families.
\end{itemize}

\end{thm}
The stability results, together with the fact that our theorem holds for all $p \le 1/2$, allow us to deduce a theorem
concerning odd-cycle-intersecting families of graphs on $n$ vertices with precisely $M$ edges, for $ M < \frac{1}{2} \C{n}{2}$.
\begin{cor}\label{corM}
Let $\alpha < 1/2$ and let  $M = \alpha \C{n}{2}$. Let $\cF$ be an odd-cycle-intersecting family of graphs
on $n$ vertices with $M$ edges each. Then
$$ |\cF| \le \C{\C{n}{2} - 3}{M -3 }.$$
Equality holds if and only if
$\cF$ is the set of all graphs with $M$ edges containing a fixed triangle.
Furthermore, if $|\cF| > (1-\eps) \C{\C{n}{2} - 3}{M -3 }$, then there exists a triangle $T$
such that all but at most $c \cdot \eps |\cF|$ of the graphs in $\cF$ contain $T$, where $c = c(\alpha)$.
\end{cor}
This corollary follows in the footsteps of Corollary 1.7 in \cite{ehud}, and we omit its
proof, since it is identical to the proof given there. It suffices to say that the idea of the proof
is to study the family of all graphs containing a graph from $\cF$, and to apply Theorem~\ref{main} to it,
together with some Chernoff-type concentration of measure results.

We are also able to generalize our main theorem in the following manner, to not necessarily uniform hypergraphs, although
we will state the theorem in terms of characteristic vectors. We discovered this generalization while studying
the question of families of subsets of $\mathbb{Z}_{2}^n$ such that the intersection of any two subsets contains a {\em Schur triple}, $\{x,y,x+y\}$.

\begin{dfn}
We say that a family \(\mathcal{F}\) of hypergraphs on \([n]\) is {\em odd-linear-dependency-intersecting} if for any \(G,H \in \mathcal{F}\) there exist \(l \in \mathbb{N}\) and nonempty sets \(A_1,A_2,\ldots,A_{2l+1} \in G \cap H\) such that
\[A_1 \triangle A_2 \triangle \ldots \triangle A_{2l+1} = \emptyset.\]
\end{dfn}

Identifying subsets of \([n]\) with their characteristic vectors in \(\{0,1\}^{n} = \mathbb{Z}_{2}^{n}\), we have the following equivalent definition:

\begin{dfn}
A family \(\mathcal{F}\) of subsets of \(\mathbb{Z}_{2}^{n}\) is {\em odd-linear-dependency-intersecting} if for any two subsets \(S,T \in \mathcal{F}\), there exist \(l \in \mathbb{N}\) and non-zero vectors \(v_1,v_2,\ldots,v_{2l+1} \in S \cap T\) such that
\[v_1 + v_2 + \ldots + v_{2l+1} = 0.\]
\end{dfn}

Naturally, an {\em odd-linear-dependency-agreeing} family is defined as above, with \(\overline{G \Delta H}\) replacing \(G \cap H\), and \(\overline{S \Delta T}\) replacing \(S \cap T\).

Note that a Schur triple is a linearly dependent set of size 3, so a Schur-triple-intersecting family is odd-linear-dependency-intersecting. We say that a family \(\mathcal{F}\) of subsets of \(\mathbb{Z}_{2}^{n}\) is a {\em Schur-umvirate} if there exists a Schur triple of non-zero vectors \(\{x,y,x+y\}\) such that \(\mathcal{F}\) consists of all subsets of \(\mathbb{Z}_{2}^{n}\) containing \(\{x,y,x+y\}\). We say that \(\mathcal{F}\) is a {\em Schur junta} if there exists a Schur triple $\{x,y,x+y\}$ such that $\mathcal{F}$ consists of all subsets of $\mathbb{Z}_{2}^{n}$ with prescribed intersection with $\{x,y,x+y\}$.

The definition of $\mu_p$ generalizes to families of subsets of $\mathbb{Z}_2^n$ in the obvious way. We have the following:

\begin{thm}\label{thmSchurIntro}
Let \(p \leq 1/2\), and let $\cF$ be an odd-linear-dependency-intersecting family of subsets of \(\mathbb{Z}_{2}^{n}\). Then
\[\mu_p(\mathcal{F}) \leq p^3.\]
Equality holds if and only if \(\mathcal{F}\) is a Schur-umvirate. Moreover, for each $p \leq 1/2$ there exists a constant \(c_p\) (bounded for $p \in [\delta,1/2]$, for any fixed \(\delta >0\)) such that for any \(\epsilon >0\), if \(\mu_p(\mathcal{F}) \geq p^3-\epsilon\) then there exists a Schur-umvirate \(\mathcal{T}\) such that
\[\mu_p(\mathcal{T} \triangle \mathcal{F}) \le c_p \epsilon.\]
For $p=1/2$, the corresponding statements hold for odd-linear-dependency-agreeing families.
\end{thm}

{\bf Remarks}: Note that this is indeed a generalization, since any triangle-intersecting family of graphs
can be lifted to a Schur-triple-intersecting family of hypergraphs by replacing every graph with $2^{2^n-{\binom{n}{2}}}$ hypergraphs in the obvious manner. In some ways, the proof of this version is simpler and more elegant. The fact that the ground set here is itself a vector space over \(\mathbb{Z}_2\)
highlights the fact that
a triangle is not only a `triangle', but in fact an `odd linear dependency over
\(\mathbb{Z}_{2}\)'. This makes the use of discrete Fourier analysis, which by design captures parity issues, a natural choice.

Note that \(\{0,1\}^{n}\) can be viewed as a vector matroid over \(\mathbb{Z}_2\). Any odd linear dependency
\[v_1+v_2\ldots+v_{2l+1} = 0\]
of non-zero vectors in \(\{0,1\}^{n}\) contains a minimal odd linear dependency, i.e. an odd-sized circuit in the matroid. Hence, Theorem \ref{thmSchurIntro} can be seen as dealing with odd-circuit-intersecting families in a matroid over \(\mathbb{Z}_{2}\).

We will defer the proof of Theorem~\ref{thmSchurIntro} until section~\ref{sectionSchur}, and concentrate on the
graph setting, which is easier to explain and to follow.

\subsection{History}
We referred above to the question of Simonovits and S\'{o}s as `beautiful'. For us this realization comes from
studying their problem intensively, and realizing that the elementary combinatorial methods (e.g. shifting)
that are often applied to Erd\H{o}s-Ko-Rado type problems do not work in this setting, and that the structure
on the ground set affects the nature of the question substantially.
The main breakthroughs in this problem, the result of \cite{cgfs} that we expand below and the current paper,
came from introducing more sophisticated machinery, which in retrospect seems to indicate the tastefulness
of the question.

There are two papers that we wish to mention in the prologue to our work, in order to sharpen our perspective.
The main progress on the Simonovits-S\'{o}s conjecture since it was posed was made in
\cite{cgfs}, where it was proved that if $\cF$ is a triangle-intersecting family, then $\mu(\cF) \leq 1/4$.
This improves upon the trivial bound of 1/2 (which follows from that fact that a graph and its complement cannot both be in
the family.) The method used in \cite{cgfs} is that of entropy/projections, and this is where the lemma known as Shearer's
entropy lemma is first stated.
It is quite interesting that our methods, under a certain restriction, also give the
bound of 1/4, although we do not see a direct connection (see Section \ref{subsec:entropyconnection}). However, the trivial observation that is our starting
point is common with \cite{cgfs}: given a triangle intersecting family $\cF$ and a bipartite graph $B$,
for any two graphs $F_1, F_2 \in \cF$ it holds that $F_1 \cap F_2$ must have a non-empty intersection with $\barB$, as a triangle cannot be contained in a bipartite graph. The approach in
\cite{cgfs} was to study the projections of $\cF$ on various graphs $\barB$ (the choice they made was taking $B$ to be
a complete bipartite graph). We will also use this observation and study intersections with various
choices of $B$, but from a slightly different angle.

Here are several remarks relevant to \cite{cgfs} that are quite useful in the current paper.
\begin{itemize}
\item The proof given in \cite{cgfs} used the fact that $B$ is bipartite, not only triangle-free,
hence it actually holds if triangle-intersecting is replaced by odd-cycle-intersecting.
This will be true of our proof too.
\item In \cite{cgfs}, it was observed that given a triangle-agreeing family, one can, by a series of monotone shifts, transform
it into a triangle-intersecting family of the same size (see section \ref{subsec:equivalence}). Hence, the maximum size of a triangle-agreeing family is equal to the maximum size of a triangle-intersecting family. In fact, the proof in \cite{cgfs} also goes through for odd-cycle-agreeing families. The same will be true of our proof, in the uniform measure case \((p=1/2)\).
\item A different way of stating the basic observation is that if $G \in \cF$ and $B$ is a bipartite
graph then
\beq{GplusH}
G \in \cF \Rightarrow (G \oplus \barB ) \not \in \cF .
\enq
This immediately suggests working in the group setting, and replacing `intersecting' with `agreeing'.
\item Although the uniform measure is perhaps the most natural one to study, the question makes perfect sense for
any measure on the subgraphs of $K_n$, specifically for the probability measure $\mu_p$
induced by the random graph model $G(n,p)$, defined above. The proof in \cite{cgfs} can be modified to give the
bound $p^2$ for any $p \le 1/2$. We improve this to $p^3$, and conjecture that this
holds for any $p \le 3/4$ (see the open problems section at the end of this paper).
\end{itemize}

A second paper that is a thematic forerunner of the current one is \cite{ehud}. It deals with the question
of the largest measure of $t$-intersecting families, using spectral methods. The immediate generalization of the
EKR theorem, appearing already in \cite{EKR}, is the case of $t$-intersecting families. For any fixed integer
$t \geq 2$, we say that a family of subsets of $[n]$ is $t$-{\em intersecting} if the intersection of any two
members of the family has size at least $t$. EKR showed that for any $k$, if $n$ is sufficiently large depending on \(k\), then the
unique largest $t$-intersecting families of $k$-subsets of $[n]$ are those obtained by taking all \(k\)-subsets containing $t$ specific elements.
Note that this is not necessarily true for smaller values of $n$, where a better construction can be, for example,
all subsets containing at least $t+1$ elements from a fixed set of $t+2$ elements. In their
paper \cite{ak}, appropriately titled `The complete intersection theorem for systems
of finite sets', Ahlswede and Khachatrian characterized the largest $t$-intersecting families for every value of \(n\), \(k\) and \(t\).

In \cite{ehud}, the question of $t$-intersecting families is studied in the setting of the product measure
of the Boolean lattice $\mathcal{P}([n])$. Let \(p \in (0,1)\) be fixed. A \(p\)-{\em random} subset of \([n]\) is a random subset of \([n]\) produced by selecting each \(i \in [n]\) independently at random with probability \(p\). We define the product measure $\mu_p$ on \(\mathcal{P}([n])\)
as follows. For any set $S \subset [n]$
define
$$\mu_p(S)=p^{|S|}(1-p)^{n-|S|},$$
i.e. the probability that a \(p\)-random subset of \([n]\) is equal to \(S\). For a family \(\F\) of subsets of \([n]\), we define
$$\mu_p(\cF)= \sum_{S \in \cF} \mu_p(S).$$
It is well known that for $p \leq 1/2$, the largest possible measure of an intersecting family is $p$,
and for $p < 1/2$, the unique largest-measure families consist of all sets containing a given element.
For $t \ge 1$ it is shown in \cite{ehud} that for $p < \frac {1}{t+1}$, the unique largest-measure $t$-intersecting
families are $t$-umvirates, the families of sets defined by containing $t$ fixed elements. Stability results
are also proved. From this it follows immediately, for example, that Theorem~\ref{main} holds for all \(p \in (0, 1/4)\) if the constant \(c\) is allowed to depend on \(p\). In the following subsection, we will discuss the relevance of the methods of \cite{ehud} to our paper.

\subsection{Methods}\label{methods}
The reason we mention \cite{ehud} in our prologue is that we are following the path set there, of applying
an eigenvalue approach to an intersection problem (and skew Fourier analysis for the non-uniform
measure). These spectral methods appear in similar settings in several much earlier papers
(e.g. \cite{Hoffman,Wilson}, to mention a few), but here they are tailored to our needs in a manner
that is inspired by \cite{ehud}.  In what follows below, we introduce at a pedestrian pace the spectral engine
that carries the proof.

Let us return to equation (\ref{GplusH}). If $\cF$ is a triangle-intersecting family
(or even an odd-cycle-agreeing family) and $B$ is a bipartite graph then we have
$$
G \in \cF \Rightarrow (G \oplus \barB) \not \in \cF .
$$
So flipping the edges of $\barB$ takes a graph in the family and produces a graph not in the family. Let us lift this operation to an operator \(A_{B}\) acting on functions whose domain is the set of subgraphs of \(K_n\), or equivalently, $\{0,1\}^{[n]^{(2)}} = \mathbb{Z}_2^{[n]^{(2)}}$. The definition is simple:
$$
A_{B}f(G) = f(G+\barB).
$$
Of course, this works equally well if we choose \(B\) at random from some distribution $\B$ over bipartite graphs, producing an operator which is an average of \(A_{B}\)'s:
$$
A_{\B}f(G) = \Exp[f(G+\barB)],
$$
where the expectation is over a random choice of $B$ from $\B$.

The important property of $A_{\B}$ for us is that if $f$ is the characteristic function of $\cF$, then whenever
$f(G)=1$, we have $A_{\B}f(G)=0$, so $f \cdot A_{\B}(f) \equiv 0$, and in particular
$$
\ip{f,A_{\B}f}=0.
$$
Now, of course, we can do this for any appropriate choice of $\B$, and take (not necessarily positive) linear
combinations of several such operators, i.e. define an operator $A$ of the type
$
A(f) = \sum c_{\B} A_{\B}(f).
$
Clearly $A$ too has the property that
\beq{ip0}
\ip{f,Af}=0.
\enq

The next step is to identify the eigenvalues and eigenfunctions of
$A$ and use equation (\ref{ip0}) to extract information about the
Fourier transform of $f$, and ultimately deduce information about
$\cF$. This eigenvalue approach in such a context stems, most
probably, from Hoffman's bounds on the size of an independent set in
a regular graph, \cite{Hoffman}. The extension we apply to deduce
uniqueness and stability is essentially reproducing the exposition
of \cite{ehud} in our setting.

It turns out that when the distribution $\B$ is easy to
understand then the spectral properties of $A_{\B}$ are also extremely easy to describe,
and most fortunately, for every choice of $\B$ one has the precise same set of eigenvectors
(whose eigenvalues depend on $\B$),
making the linear combination $\sum c_{\B} A_{\B}$ particularly easy to understand and analyze.

Finally, in one sentence, we explain why fourteen years passed between the moment in which Vera S\'{o}s asked the
third author the question treated in this paper, and the resolution of the problem: even after discovering
the spectral path, how does one choose the
distributions $\B$ and the appropriate weights $c_{\B}$ in a way which produces the correct eigenvalues?
Most of the paper deals with the answer to that question.

\subsection{Structure of the paper}
We will treat the cases of $p=1/2$ and $p < 1/2$ separately, since the latter is slightly more complex and less routine.
In section \ref{sec:phalf} we begin the case of $p=1/2$, and describe the main tools that we will use for the proof. In section \ref{secOCC} we construct the operators and spectra that prove our main theorem. In section \ref{sectionCutStatistics} we study the cut  statistics of random cuts of a graph, and prove the necessary facts
that show that our operators have the desired properties. In section \ref{secSmallp} we treat the case of $p < 1/2$.
In section \ref{sectionSchur} we prove the more general theorem on Schur-triple-intersecting families. In section \ref{secDiscussion} we conclude with some related open problems.
\section{The uniform measure, \texorpdfstring{$p=1/2$}{p = 1/2}} \label{sec:phalf}
\subsection{Fourier Analysis}
\label{subsec:fourieranalysis}
We briefly recall the essentials of Fourier Analysis on the Abelian group \(\mathbb{Z}_2^{X}\), where \(X\) is a finite set. (In our case, the set \(X\) will usually be \([n]^{(2)}\), the edge-set of the complete graph \(K_{n}\), and subsets \(S \subset X\) will be replaced by subgraphs \(G \subset K_n\).) We identify \(\mathbb{Z}_2^{X}\) with the power-set of \(X\) in the natural way, i.e. a subset of \(X\) corresponds to its characteristic function.

For any two functions $f, g : \mathbb{Z}_2^{X} \to \mathbb{R}$, we define their inner product as
\[\inner{f}{g} = \Exp(f \cdot g) = \frac{1}{2^{|X|}}\sum_{S \subset X} f(S) g(S);\]
this makes \(\mathbb{R}[\mathbb{Z}_2^{X}]\) into an inner-product space. For every subset $R \subset X$, we define a function
$\chi_R \colon \mathbb{Z}_2^{X} \ra \mathbb{R}$ by
$$ \chi_R(S) = (-1)^{|R \cap S|}.$$
Then \(\chi_R\) is a character of the group \(\mathbb{Z}_2^{X}\), since for any $S,T \subset X$, we clearly have
$$
\chi_R(S\oplus T) = \chi_R(S)\cdot \chi_R(T).
$$
It is routine to verify that the set \(\{\chi_{R}:\ R \subset X\}\) is an orthonormal basis for the vector space \(\mathbb{R}[\mathbb{Z}_2^{X}]\) of all real-valued functions on \(\mathbb{Z}_{2}^{X}\); it is called the {\em Fourier-Walsh basis}. Hence, every
$f \colon \mathbb{Z}_2^{X} \ra \mathbb{R}$ has a unique expansion of the form
\begin{equation}
 \label{eq:fourier}
f = \sum_{R \subset X} \widehat{f}(R)\chi_R;
\end{equation}
we have $\widehat{f}(R) = \ip{f,\chi_R}$. We call (\ref{eq:fourier}) the {\em Fourier expansion} of \(f\). From orthonormality, for any two functions $f, g$, we have {\em Parseval's Identity}:
$$
\ip{f,g}= \sum_{R \subset X} \widehat{f}(R)\widehat{g}(R).
$$
In particular, whenever $f$ is Boolean ($0/1$ valued), taking \(g \equiv 1\) gives:
$$
\widehat{f}(\emptyset) = \ip{f,\bf{1}} = \Exp{[f]} = \Exp{[f^2]} = \ip{f,f} = \sum_{R \subset X} \widehat{f}^2(R).
$$
Abusing notation, we will let $\cF$ denote both a family of sets and its characteristic function,
so the above will be used in the form
$$
\mu(\cF)= \widehat{\cF}(\emptyset) = \sum_{R \subset X} \widehat{\cF}^2(R).
$$
Another formula that is useful to keep in the back of our minds is the convolution formula:
$$
\widehat{f*g} = \widehat{f}\cdot\widehat{g},
$$
where $f*g(S)=\sum_{T \subset X} f(T)g(S+T).$

\subsection{Cayley operators and their spectra}
Questions about largest intersecting families can often be translated into the question of finding a largest
independent set in an appropriate graph (often a Cayley graph). One can then use the spectral approach
due to Hoffman \cite{Hoffman} to bound the size of the largest independent set in terms of the eigenvalues
of the graph (meaning the eigenvalues of its adjacency matrix). A central idea in \cite{EFP} and \cite{ehud} is that
one may choose appropriate weights on the edges of this graph to perturb the
operator defined by the adjacency matrix, and improve these bounds. These weights need not
necessarily be positive. In this paper, we will call these perturbed operators Odd-Cycle-Cayley operators,
or OCC operators for short. The Cayley graph \(\Gamma\) that we have is on the group $\mathbb{Z}_2^{[n]^{(2)}}$, with the set of
generators consisting of all graphs $\barB$ such that $B$ is a bipartite graph,
\[\Gamma = \Gamma\left(\mathbb{Z}_2^{[n]^{(2)}},\{\bar{B}:\ B \subset K_n,\ B \textrm{ is bipartite}\}\right).\]
 Note that an odd-cycle-agreeing family of subgraphs of $K_n$ is precisely an independent set in this graph.

\begin{dfn}\label{OCC}
A linear operator $A$ on real-valued functions on $\mathbb{Z}_2^{[n]^{(2)}}$ will be called {\em Odd-Cycle-Cayley},
or {\em OCC} for short, if it has the following two properties:
\begin{enumerate}
\item If $\cF$ is an odd-cycle-agreeing family, and $f$ is its characteristic function, then
\[f(G)=1 \Rightarrow Af(G)=0.\]
\item The Fourier-Walsh basis is a (complete) set of eigenfunctions of $A$.
\end{enumerate}
For each \(G \subset K_n\), we write \(\lambda_{G}\) for the eigenvalue corresponding to the eigenfunction \(\chi_{G}\). We write \(\Lambda = (\lambda_{G})_{G \subset K_n}\) for the vector of eigenvalues of the OCC operator; we call this an {\em OCC spectrum}. We denote the
minimum eigenvalue by $\lambda_{\min}$, and we write $\Lambda_{min}$ for the set of graphs \(G\) with $\lambda_{G} = \lambda_{\min}$; we will call these the `{\em tight graphs}'. The {\em spectral gap} of $\Lambda$ is the maximal $\gamma$
such that $\lambda_H \ge \lambda_{\min} + \gamma$ for all $H \not \in \Lambda_{\min}$.
\end{dfn}
Note that the set of OCC operators forms a linear space, and hence
also the set of OCC spectra is a linear space, a fact that is of
crucial importance for us.

Our main tool for constructing OCC operators is by using Equation \ref{GplusH} as described in subsection
\ref{methods} where we discussed our methods. Let \(B\) be a bipartite graph, and let \(A_{B}\) be the operator on real-valued functions on \(\mathbb{Z}_{2}^{[n]^{(2)}}\), defined by
$$
A_{B}f(G) = f(G+\barB).
$$
Similarly, let $\B$ be a distribution over bipartite graphs, and let
$$
A_{\B}f(G) = \Exp[f(G \oplus \barB)],
$$
where the expectation is over a choice of $B$ from $\B$. We make the following
\begin{claim}\label{claimOCC}
$A_{\B}$ is an OCC operator, and its spectrum is given by
$$
\lambda_R = (-1)^{|R|} \Exp[\chi_B(R)].
$$
\end{claim}
Before proving the claim, we list several equivalent ways of describing $A_{\B}$, depending on one's mathematical taste:
\begin{itemize}
\item $A_{\B}$ is a convolution operator, and therefore has the elements of the Fourier-Walsh basis as eigenfunctions.
\item $A_{\B}$ is the average of operators \(A_{B}\). Note that \(A_{B}\)
is a tensor product of $\binom{n}{2}$ operators (one for each edge of \(K_{n}\)), each acting on functions on a two-point space. Hence, the eigenfunctions of each \(A_{B}\) include the tensor products of the eigenfunctions from each coordinate,
which, again, is the Fourier-Walsh basis. Therefore, the same is true of \(A_{\B}\).
\item Alternatively, note that \(A_{B}\) is the operator defined by the adjacency matrix of the Cayley graph on \(\mathbb{Z}_{2}^{[n]^{(2)}}\) with generating set \(\{\bar{B}\}\), which is a subgraph of \(\Gamma\). \remove{Let \(\{\mathbf{e}_{G}:\ G \subset K_n\}\) denote the standard basis of \(\mathbb{R}[\mathbb{Z}_{2}^{[n]^{(2)}}]\), where \(\mathbf{e}_{G}\) denotes the function which is 1 at \(G\) and \(0\) elsewhere. The matrix of \(A_{B}\) with respect to the standard basis is simply the adjacency matrix of the Cayley graph on \(\mathbb{Z}_{2}^{[n]^{(2)}}\) with generating set \(\{\bar{B}\}\).} It is well-known that the eigenvectors of the adjacency matrix of any Cayley graph on an Abelian group include the characters of the group, i.e. the Fourier-Walsh basis in our case.
\item $A_{\B}$ is a Markov operator describing a random walk on $\mathbb{Z}_2^{[n]^{(2)}}$.
This random walk has the uniform measure as its stationary measure and has the property that
if $\cF$ is odd-cycle-intersecting then two consecutive steps cannot both lie in $\cF$.
\end{itemize}
This last characterization, which may seem less appealing, will become quite illuminating once we move
to the setting of $\mu_p$ for $p < 1/2$.
\begin{proof}[of Claim \ref{claimOCC}]
It is clear that if $\cF$ is an odd-cycle-agreeing family, and $f$ its characteristic function then
$$
f(G)=1 \Rightarrow A_{\B}f(G)=0.
$$
It is also quite simple to verify that the Fourier-Walsh characters are eigenfunctions of $A_{\B}$, and to
give an explicit formula for the eigenvalues:
$$
A_{\B}\chi_R(G)=\Exp[\chi_R(G \oplus \barB)]=\chi_R(G)\cdot\Exp[\chi_R(\barB)],
$$
hence
$$
\lambda_R=\Exp[\chi_R(\barB)].
$$
It turns out to be slightly more useful to write this last expression as given by our claim:
\beq{eigenvalue}
\lambda_R=(-1)^{|R|}\Exp[\chi_B(R)] = (-1)^{|R|}\Exp[\chi_B(R \cap B)].
\enq
\end{proof}
The following theorem is a weighted version of Hoffman's theorem \cite{Hoffman} which bounds the size of an independent
set in a regular graph in terms of its eigenvalues.
\begin{thm}\label{ThmHoffman}
Let $\Lambda=(\lambda_G)_{G \subset K_n}$ be an OCC spectrum with \(\lambda_{\emptyset} = 1\), with minimal value
$\lambda_{\min}$ such that $-1 < \lambda_{\min} < 0$, and with spectral gap $\gamma>0$. Set $\nu = \frac{-\lambda_{\min}}{1-\lambda_{\min}}$ (so $\lambda_{\min}= \frac{-\nu}{1-\nu}$).
 Then for any odd-cycle-agreeing family $\cF$ of subgraphs of $K_n$ the following holds:
\begin{itemize}
 \item Upper bound: $\mu(\cF) \leq \nu$.
 \item Uniqueness: If $\mu(\cF) = \nu$ then
     $\widehat{\cF}(G)\not = 0$ only for $G \in \Lambda_{\min} \cup \{\emptyset\}$.
 \item Stability: Let $w = \sum_{G \not \in \Lambda_{\min} \cup
     \{\emptyset\}}\widehat{\cF}^2(G)$. Then $w \leq
      \frac{\nu}{(1-\nu)\gamma} (\nu-\mu(\cF))= O(\nu -\mu(\cF)).$
\end{itemize}
 \end{thm}
 Before proving this theorem, let us state a corollary which will be
 the form in which the theorem is applied.
\begin{cor}\label{corMain}
Suppose that there exists an OCC spectrum $\Lambda$ with eigenvalues $\lambda_{\emptyset}=1$,
$\lambda_{\min}=-1/7$ and spectral gap $\gamma >0$. Assume that all
graphs in $\Lambda_{\min}$ (the set of graphs $G$ for which
$\lambda_G=\lambda_{\min}$) have at most $3$ edges. Then if $\cF$ is
an odd-cycle-agreeing family of subgraphs of $K_n$ it holds that
\begin{itemize}
\item Upper bound: $\mu(\cF) \leq 1/8$.
 \item Uniqueness: If $\mu(\cF) = 1/8$, then $\cF$ is a triangle junta.
 \item Stability: If $\mu(\cF) \geq 1/8 - \epsilon$, then there exists a triangle junta $\cT$ such that
 $\mu(\cF \Delta \cT) \leq c\epsilon$, where \(c>0\) is an absolute constant.
\end{itemize}
\end{cor}
\begin{proof}[of Theorem \ref{ThmHoffman}]
Let \(A\) be an OCC operator with spectrum \(\Lambda\); then
\[A(\cF) = \sum_{G} \lambda_G \widehat{\cF}(G)\chi_G,\]
and therefore
$$
0=\ip{\cF,A\cF} = \sum_G \lambda_G \widehat{\cF}^2(G).
$$
Next, recall that $\widehat{\cF}(\emptyset) = \sum_G \widehat{\cF}^2(G) = \mu(\cF)$.
Since
$$
w = \sum_{G \not \in \Lambda_{\min} \cup \{\emptyset\}}\widehat{\cF}(G)^2,
$$
we have
\[\sum_{G \in \Lambda_{\min}} \widehat{\cF}(G)^2 = \mu(\mathcal{F})-\mu(\cF)^2 - w.\]
Hence,
\begin{align*}
0= \sum_G \lambda_G \widehat{\cF}^2(G) &\ge \lambda_{\emptyset} \mu(\cF)^2 + \lambda_{\min} (\mu(\mathcal{F})-\mu(\cF)^2-w)+(\lambda_{\min} +\gamma)w\\
& = \mu(\cF)^2 -\frac{\nu}{1-\nu} (\mu(\cF) - \mu(\cF)^2 - w)
+ w\left( \gamma - \frac{\nu}{1-\nu}\right) \\
&= \frac{\mu(\cF)^2}{1 - \nu} - \frac{\mu(\cF)\nu}{1 - \nu} + w \gamma.
\end{align*}
Therefore,
\[\mu(\cF)^2 - \mu(\cF) \nu + w\gamma (1 - \nu) \geq 0.\]
Since $\gamma > 0$, we immediately obtain $\mu(\cF) \leq \nu$, with equality if and only if $w = 0$. Thus,
\[ \mu(\cF)^2 - \mu(\cF) \nu + w\gamma (1 - \nu) \frac{\mu(\cF)}{\nu} \geq 0.\]
Cancelling and rearranging, we obtain:
\[w \leq \frac{\nu}{(1-\nu)\gamma} (\nu-\mu(\cF)),\]
as required.
\end{proof}
\begin{proof}[of Corollary \ref{corMain}]
The upper bound of $1/8$ follows immediately from Theorem~\ref{ThmHoffman}. The uniqueness claim is
a special case of \cite[Lemma~2.8(1)]{ehud}, which we will quote below. The stability follows
from a powerful result of Kindler and Safra \cite{kindler-safra} (as was the case in \cite{ehud}.)
We recall their result too.

$\bullet$ {\bf Uniqueness.} We first prove the uniqueness under the assumption that $\cF$
is odd-cycle-intersecting. The reduction from the agreeing case to the intersecting case
is done in Lemma~\ref{agree2intersect} in the following subsection.

Since we know that the Fourier transform of $\cF$ is concentrated on graphs with
at most 3 edges, it follows from a result of Nisan and Szegedy \cite[Theorem~2.1]{NisanSzegedy}
that $\cF$ depends on at most $3\cdot2^3=24$ coordinates, i.e. can be described by the intersection of its
members with a graph on 24 edges. However, even with a computer it seems extremely difficult to check
all such examples. Luckily for us we have two additional assumptions. First, we may assume that $\cF$ is an up-set,
else we can replace it by its up-filter, the family of all graphs containing a member of $\cF$, which would preserve
the intersection property. Secondly, we have $\mu(\cF)=1/8$.
This falls precisely into the setting of the following lemma.
\begin{lem}[\cite{ehud}~]
\label{lem:ehud}
 Let \(N \in \mathbb{N}\), let $p \le 1/2$ and suppose $f:\{0,1\}^{N} \to \{0,1\}$ is a monotone Boolean function with
  $\Exp_p f = p^t$, and $\widehat{f}(S) = 0$ whenever $|S|>t$.
   Then $f$ is a $t$-umvirate (depends only on~$t$ coordinates).
\end{lem}
(Here, the expectation $\Exp_p$ is taken with respect to the skew product measure \(\mu_{p}\) on \(\{0,1\}^{N}\); in our case, $p=1/2$.) Clearly, in our case, if $\cF$ is triangle-intersecting and a 3-umvirate, it is a
$\trium$. The reduction from odd-cycle-agreeing families to odd-cycle-intersecting families is in Lemma \ref{agree2intersect}

$\bullet$ {\bf Stability.} We need Theorem 3 from \cite{kindler-safra}:
\begin{thm}[Kindler-Safra]
\label{thm:kindlersafra}
 For every $t \in \mathbb{N}$, there exist $\epsilon_0>0$, $c_{0} >0$ and $T_{0} \in \mathbb{N}$ such that the following holds. Let \(N \in \mathbb{N}\), and let $f:\{0,1\}^{N} \to \{0,1\}$ be a Boolean function such that
 $$
 \sum_{|S|>t} \widehat{f}(S)^2 = \varepsilon < \varepsilon_0.
 $$
 Then there exists a Boolean function $g:\{0,1\}^{N} \to \{0,1\}$, depending on at most $T_{0}$ coordinates, such that
 $$
 \mu(\{R: f(R)\not = g(R)\} ) \leq c_{0} \varepsilon.
 $$
\end{thm}
The nice thing about this theorem is that as soon as the Fourier weight on the higher levels is small
 enough, the number of coordinates needed for the approximating family does not grow.
We apply this in our setting as follows. Assume that \(\mathcal{F}\) is an odd-cycle-agreeing family of subgraphs of \(K_n\) with $\mu(\cF) > 1/8 - \epsilon$. From Theorem \ref{ThmHoffman}, we have
$$
w = \sum_{G \not \in \Lambda_{\min} \cup \{\emptyset\}}\widehat{\cF}(G)^2 \leq \frac{1}{7 \gamma}\epsilon.
$$
Applying the Kindler-Safra result with \(t=3\), we see that provided $\epsilon \leq 7 \gamma \epsilon_0$, $\cF$
is $(\frac{c_0}{7\gamma}  \epsilon)$-close to some family $\cG$ depending on a set $A$ of at most $T_0$ coordinates (edges).
Moreover, as we show below, if $\frac{c_0}{7\gamma}\epsilon < 2^{-T_0}$, then $\cG$ is odd-cycle-agreeing. But there are only a finite number of such families that are not triangle juntas, and by our uniqueness result, all have measure less than $1/8$. Choose \(\epsilon_1 > 0\) such that all of these families have measure less than $1/8 - (1+\frac{c_0}{7 \gamma})\epsilon_1$. If \(\epsilon \leq \epsilon_1\), $\cF$ cannot have measure at least $1/8 -\epsilon$ and be $(\frac{c_0}{7 \gamma}  \epsilon)$-close to one of these families,
so the approximating family guaranteed
by Kindler-Safra must be a triangle junta. If \(\epsilon \geq \min(7 \gamma \epsilon_0,\epsilon_1,\frac{7\gamma}{c_0}2^{-T_0}) =: \epsilon_2\), we may simply choose the constant \(c = 1/\epsilon_2\), completing the proof of Corollary \ref{corMain}.

It remains to show that if $\frac{c_0}{7\gamma}\epsilon < 2^{-T_0}$ then $\cG$ is odd-cycle-agreeing. Suppose that $\cG$ contained two graphs $G_1,G_2$ supported on $A$ which aren't odd-cycle-agreeing. Let $\cF_1,\cF_2 \subset \cF$ consist of those graphs in $\cF$ whose restriction to $A$ is $G_1,G_2$ (respectively). If $\mu(\cF_1) + \mu(\cF_2) > 2^{-|A|}$ then there must exist two graphs $J_1,J_2$ forming a partition of $\overline{A}$ such that $H_1 \eqdef G_1 \cup J_1 \in \cF_1$ and $H_2 \eqdef G_2 \cup J_2 \in \cF_2$, and so $\cF$ contains two graphs $H_1,H_2$ whose agreement is $\overline{H_1\oplus H_2} = \overline{G_1\oplus G_2}$. Since $\cF$ is odd-cycle-agreeing, this cannot happen, and we deduce that $\mu(\cF_1) + \mu(\cF_2) \leq 2^{-|A|}$, which implies that the distance between $\cF$ and $\cG$ is at least $2^{-|A|} \geq 2^{-T_0}$, contrary to assumption.
\end{proof}
\subsection{The intersecting / agreeing equivalence}
\label{subsec:equivalence}
In the proof of the uniqueness statement in Corollary \ref{corMain} we assumed that the family of graphs in question was
odd-cycle-intersecting. We now wish to reduce the general case of odd-cycle-agreeing to that of odd-cycle-intersecting. To this end, it will be helpful to return to the related observation of Chung, Frankl, Graham and Shearer in \cite{cgfs} mentioned earlier. For completeness, we reproduce their general statement and proof, as we will wish to build upon it.

Let \(X\) be a finite set, and let \(\mathcal{Z} \subset \mathcal{P}(X)\) be a family of subsets of \(X\). We say that a family \(\mathcal{F} \subset \mathcal{P}(X)\) is \(\mathcal{Z}\)-\emph{intersecting} if for any \(A,B \in \mathcal{F}\) there exists \(Z \in \mathcal{Z}\) such that \(Z \subset A \cap B\). We say that \(\mathcal{F} \subset \mathcal{P}(X)\) is \(\mathcal{Z}\)-\emph{agreeing} if for any \(A,B \in \mathcal{F}\) there exists \(Z \in \mathcal{Z}\) such that \(Z \cap (A \Delta B) = \emptyset\). We write
\[m(\mathcal{Z}) = \max\{|\mathcal{A}|:\ \mathcal{A} \subset \mathbb{P}X,\ \mathcal{A} \textrm{ is } \mathcal{Z}\textrm{-intersecting}\}\]
and
\[\overline{m}(\mathcal{Z}) = \max\{|\mathcal{A}|:\ \mathcal{A} \subset \mathbb{P}X,\ \mathcal{A} \textrm{ is } \mathcal{Z}\textrm{-agreeing}\}.\]
Chung, Frankl, Graham and Shearer proved the following:
\begin{lem}
\label{lemma:equivalence}
Let \(X\) be a finite set, and let \(\mathcal{Z} \subset \mathcal{P}(X)\). Then \(\overline{m}(\mathcal{Z}) = m(\mathcal{Z})\).
\end{lem}
\begin{proof}
 Clearly, a \(\mathcal{Z}\)-intersecting family is \(\mathcal{Z}\)-agreeing, and therefore \(m(\mathcal{Z}) \leq \overline{m}(\mathcal{Z})\). We will show that any \(\mathcal{Z}\)-agreeing family can be made into a \(\mathcal{Z}\)-intersecting family of the same size.

For any \(i \in X\), consider the \(i\)-monotonization operation \(C_{i}\), defined as follows.
Given a family \(\mathcal{A} \subset \mathcal{P}(X)\), \(C_{i}(\mathcal{A})\) is produced by replacing \(A\) with \(A \cup \{i\}\) for each set \(A\) such that \(i \notin A\), \(A \in \mathcal{A}\) and \(A \cup \{i\} \notin \mathcal{A}\). (Note that \(C_i\) is a special case of the so-called \(UV\)-\emph{compression} \(C_{UV}\), with \(U = \{i\}\) and \(V = \emptyset\). The reader may refer to \cite{frankl} for a discussion of \(UV\)-compressions and their uses in combinatorics.)

Clearly, \(|C_{i}(\mathcal{A})| = |\mathcal{A}|\); it is easy to check that if \(\mathcal{A}\) is \(\mathcal{Z}\)-agreeing then so is \(C_{i}(\mathcal{A})\).

Now let \(\mathcal{F} \subset \mathcal{P}(X)\) be a \(\mathcal{Z}\)-agreeing family, and successively apply the operations \(C_{i}\) for \(i \in X\). Formally, we set \(\mathcal{F}_{0} = \mathcal{F}\); given \(\mathcal{F}_{k}\), if there exists \(F \in \mathcal{F}_{k}\) and \(i\in X\) such that \(F \cup \{i\} \notin \mathcal{F}_{k}\), then we let \(\mathcal{F}_{k+1} = C_{i}(\mathcal{F}_{k})\). At each stage of the process, the sum of the sizes of the sets in the family increases by at least 1, so the process must terminate, say with the family \(\mathcal{F}_{l}\). Let \(\mathcal{F}' = \mathcal{F}_{l}\). Observe that \(\mathcal{F}'\) is a \(\mathcal{Z}\)-agreeing family with \(|\mathcal{F}'| = |\mathcal{F}|\). Moreover, it is an up-set, meaning that if \(F \in \mathcal{F}'\) and \(G \supset F\), then \(G \in \mathcal{F}'\). It follows that \(\mathcal{F}'\) must be \(\mathcal{Z}\)-intersecting. (If \(F,G \in \mathcal{F}'\), then \(F \cup \overline{G} \in \mathcal{F}'\), so there exists \(Z \in \mathcal{Z}\) 
such that \(((F \cup \overline{G}) \Delta G) \cap Z = \emptyset\). But then \(F \cap G \supset Z\). Hence, \(\mathcal{F}'\) is \(\mathcal{Z}\)-intersecting.)

It follows that \(\overline{m}(\mathcal{Z}) \leq m(\mathcal{Z})\), and therefore \(\overline{m}(\mathcal{Z}) = m(\mathcal{Z})\), as required.
\end{proof}

We can now complete the proof of the uniqueness statement in Corollary \ref{corMain}, which claims
that if $\cF$ is an odd-cycle-agreeing family and $\mu(\cF)=1/8$,
then $\cF$ is a triangle junta. We will apply the monotonization operations above to $\cF$, and produce an odd-cycle-intersecting
family of the same size, which by our results must be a $\trium$.
The following lemma then shows that $\cF$ must be a triangle
junta.
\begin{lem}\label{agree2intersect}
Let $\cF$ be an odd-cycle-agreeing family, and assume that a series of monotonization operations \(C_{e}\) (for \(e \in [n]^{(2)}\)) as described above produces a family $\cF_k$ which is a $\trium$. Then $\cF$ is a triangle junta.
\end{lem}

\begin{proof}
Suppose \(\mathcal{F}_{k} \subset \mathbb{Z}_{2}^{[n]^{(2)}}\) is odd-cycle-agreeing, and \(\mathcal{F}_{k+1} = C_{e}(\mathcal{F}_{k}) \neq \mathcal{F}_{k}\) is a \(T\)-junta for some triangle \(T \subset K_n\). Then there exists a graph \(G \notin \mathcal{F}_{k+1}\) such that \(G \cup \{e\} \in \mathcal{F}_{k+1}\); since \(\mathcal{F}_{k+1}\) is a \(T\)-junta, we must have \(e \in T\). Let \(S\) be the subgraph of \(T\) such that
\[\mathcal{F}_{k+1} = \{G \in \mathbb{Z}_{2}^{[n]^{(2)}}:\ G \cap T = S\}.\]
Clearly, \(e \in S\). Let \(\mathcal{C} = \{G \in \mathcal{F}_{k}:\ e \notin G\}\), and let \(\mathcal{D} = \{G \in \mathcal{F}_{k}:\ e \in G\}\); then we may express
\[\mathcal{F}_{k} = \mathcal{C} \sqcup \mathcal{D}.\]
Observe that if \(G \in \mathcal{C}\), then \(G \cup \{e\} \notin \mathcal{D}\): if \(G \in \mathcal{C}\) and \(G \cup \{e\} \in \mathcal{D}\), then \(G,G \cup \{e\} \in \mathcal{F}_{k+1}\), contradicting the fact that all graphs in \(\mathcal{F}_{k+1}\) contain \(S\). Hence,
\[\mathcal{F}_{k+1} = \mathcal{D} \sqcup \{G \cup \{e\}:\ G \in \mathcal{C}\}.\]
It follows that all graphs \(G \in \mathcal{D}\) have \(G \cap T = S\), and all graphs \(G \in \mathcal{C}\) have \(G \cap T = S -e\). Since \(\mathcal{F}_{k+1} \neq \mathcal{F}_{k}\), we must have \(\mathcal{C} \neq \emptyset\); we will show that \(\mathcal{D} = \emptyset\). Suppose for a contradiction that \(\mathcal{D} \neq \emptyset\). Let \(|\mathcal{D}| = N \geq 1\); then \(|\mathcal{C}| = 2^{{\binom{n}{2}}-3}-N \geq 1\). Since \(\overline{T} \oplus e\) intersects every triangle, if \(G \in \mathcal{F}_{k}\) then \(G \oplus (\overline{T} \oplus e) \notin \mathcal{F}_{k}\). Since
\[|\mathcal{C}|+|\mathcal{D}| = 2^{{\binom{n}{2}}-3},\]
for every \(H \subset \overline{T}\) exactly one of \(H \oplus S\in \mathcal{D}\) and \(H\oplus S \oplus \overline{T} \oplus e \in \mathcal{C}\) holds. In other words, the classes
\[\mathcal{V} = \{H \subset \overline{T}:\ H\oplus S \in \mathcal{D}\},\quad \mathcal{W} = \{H \subset \overline{T}:\ H\oplus S\oplus \overline{T} \oplus e\in \mathcal{C}\}\]
form a partition of the set of labeled subgraphs of \(\overline{T}\), with both classes nonempty. Hence, there exist two adjacent subgraphs of \(\overline{T}\) in different classes, i.e. there exists a subgraph \(H \subset \overline{T}\) and an edge \(f \in E(\overline{T})\) such that \(H \oplus S \in \mathcal{D}\), and \(H\oplus f \oplus S \oplus \overline{T} \oplus e \in \mathcal{C}\).
 But these two graphs agree only on the graph \(T\oplus e \oplus f\), which is a 3-edge graph containing exactly two edges of the triangle \(T\), so cannot be a triangle. This contradicts our assumption that \(\mathcal{F}_{k}\) is odd-cycle-agreeing.

We may conclude that \(\mathcal{D} = \emptyset\), i.e.
\[\mathcal{F}_{k} = \{G \in \cG_{n}:\ G \cap T = S-e\}.\]
Hence, \(\mathcal{F}_{k}\) is also a \(T\)-junta.

By backwards induction on \(k\), we see that \(\mathcal{F}_{0} = \mathcal{F}\) is also a \(T\)-junta, completing the proof.
\end{proof}

\subsection{Constructing the required OCC spectrum}\label{secOCC}
In this section we prove the existence of an OCC operator with the
desired spectrum, which together with Corollary \ref{corMain} will
complete the proof of Theorem \ref{main} for the case of $p=1/2$.
Our construction will proceed in two steps. First, we prove the
existence of an OCC spectrum $\Lambda^{(1)}$ with the correct minimal
eigenvalue, but for which $\Lambda_{\min}$, the set of graphs on
which it is obtained, includes also 4-forests and $K_4^-$. We then
take care of these extra graphs by adding a multiple of $\Lambda^{(2)}$, an OCC
spectrum that takes positive value on these problematic graphs while
having value 0 for all graphs with three or less edges.

The main lemma we use is extremely easy to state and prove, yet turns out to be very useful.
\begin{lem}\label{lemOCCSpectrum}
Let $\B$ be a distribution on bipartite graphs, and for every $B \in \B$ let $f_B$ be a real-valued
function whose domain is the set of subgraphs of $B$. Then the following function is an OCC spectrum:
$$
\lambda_G = (-1)^{|G|} \Exp[f_B(B \cap G)],
$$
where, as usual, the expectation is with respect to a random choice of $B$ from $\B$.
\end{lem}
\begin{proof}
Fix a bipartite graph $B$. From Claim \ref{claimOCC}, we know that \(A_{B}\) is an OCC operator. Equivalently, from equation (\ref{eigenvalue}),
 $\lambda_G = (-1)^{|G|} \chi_B(G \cap B)$
is an OCC spectrum. Moreover, if \(B'\) is any subgraph of \(B\), the function $(-1)^{|G|} \chi_{B'}(G \cap B) = (-1)^{|G|} \chi_{B'}(G \cap B')$ also describes an OCC spectrum. Since the set
$\{\chi_{B'} : {B' \subset B}\}$ spans all functions $f$ on the
subgraphs of $B$, we see that for any choice of $f$, the vector
described by $\lambda_G = (-1)^{|G|} f(G \cap B) $ is also an OCC
spectrum. Taking expectation with respect to a random choice of $B$ from $\B$ completes the proof.
\end{proof}
The few choices of $f_{B}$ and \(\mathcal{B}\) for which we will apply this lemma are quite simple. The distribution \(\mathcal{B}\) will always be the uniform distribution on complete bipartite subgraphs of \(K_n\), and the functions \(f_{B}\) will always be invariant under isomorphism of subgraphs of \(B\). Hence, our OCC spectra \((\lambda_{G})_{G \subset K_n}\) will always be invariant under graph isomorphism, so they may be seen as functions on the set of {\em unlabelled} graphs with at most \(n\) vertices. In fact, we will choose $f_{B}(G \cap B)$ to be the indicator function of the event that the number of edges of $G \cap B$ is $i$
(for $i = 1,2$ or $3$), or to be the indicator function of $G \cap B$ being isomorphic to a given graph $R$
(for some small list of $R$'s).
\begin{cor}
\label{cor:qB}
Let $(V_1,V_2)$ be a random bipartition of the vertices of $K_n$, where each vertex is chosen independently to
belong to each $V_i$ with probability 1/2. Let $B$ be the set of edges of \(K_n\) between $V_1$ and $V_2$. For any graph $G \subseteq K_n$, let
$$
q_i(G) = \Pr[ |G \cap B|= i ],
$$
and for any bipartite graph $R$, let
$$
q_R(G) = \Pr[ (G \cap B) \isomorphicto R],
$$
where $H \isomorphicto R$ means that $H$ is isomorphic to $R$; all probabilities are over the choice of the
random bipartition. Then for any integer $i$,
$$
 \lambda_G = (-1)^{|G|} q_i(G)
$$
 is an OCC spectrum, and for any bipartite graph $R$,
$$
(-1)^{|G|} q_R(G)
$$
is an OCC spectrum.
\end{cor}
Recall that if \(G\) is a graph, a {\em cut} in \(G\) is a bipartite subgraph of \(G\) produced by partitioning the vertices of \(G\) into two classes \(V_1\) and \(V_2\), and taking all the edges of \(G\) that go between the two classes. If \(V_1,V_2\) and \(B\) are as above, \(G \cap B\) is called a {\em(uniform) random cut in} \(G\). Note that \(q_{i}(G)\) is the probability that a random cut in \(G\) has exactly \(i\) edges, so is relatively easy to analyze; \(q_{R}(G)\) is the probability that a random cut in \(G\) is isomorphic to \(R\).

The beauty of the functions $q_i(G)$ and $q_R(G)$ is that they supply us with a rich enough space of eigenvalues to create
a spectrum with the correct values on small graphs, yet they decay quickly with the size of $G$, ensuring that
the eigenvalues of larger graphs will be bounded away from $\lambda_{\min}$.
When tackling the problem, we tried taking a linear combination of as few as possible
of these building blocks, constructing an OCC spectrum that obtains the desired values on subgraphs
of the triangle; we prayed that this is feasible, and that the resulting eigenvalues for larger graphs maintain
a spectral gap. Happily, with some fine tuning, this works.
This is manifested in the following two claims.
\begin{claim}\label{goodOCC1}
Let $\Lambda^{(1)}$ be the OCC spectrum described by
$$
\lambda^{(1)}_G = (-1)^{|G|} \left[ q_0(G) -\frac{5}{7}q_1(G) -\frac{1}{7}q_2(G) + \frac{3}{28} q_3(G)\right].
$$
Then
\begin{itemize}
\item $\lambda^{(1)}_\emptyset = 1$.
\item $\lambda^{(1)}_{\min}=-1/7$.
\item $\Lambda^{(1)}_{\min}$ consists of the following graphs: a
    single edge, a path of length two, two disjoint edges, a
    triangle, all forests with four edges, and $K_4^-$.
\item For all $H \not \in \Lambda^{(1)}_{\min}$ it holds that
    $\lambda^{(1)}_H \ge -1/7 + \gamma'$, with $\gamma' = 1/56$.
\end{itemize}
\end{claim}

\begin{claim}\label{goodOCC2}
Let $\Lambda^{(2)}$ be the OCC spectrum described by
$$
\lambda^{(2)}_G = (-1)^{|G|} \left[\sum q_F(G) - q_{\Box}(G)\right]
$$
where the sum is over all 4-forests $F$, and $\Box$ denotes $C_4$.
Then \begin{enumerate}
\item $\lambda^{(2)}_H = 0$ for all $H$ with less than 4 edges.
\item $\lambda^{(2)}_F= 1/16$ for all 4-forests.
\item $\lambda^{(2)}_{K_4^-}=1/8$.
\item $|\lambda^{(2)}_G| \le 1$ for all $G$.
\end{enumerate}
\end{claim}
We defer the proof of Claim \ref{goodOCC1} to the next section where we analyze the cut statistics of a random cut of a graph.
The proof of Claim \ref{goodOCC2} is quite easy.
\begin{proof}
We follow the items of the claim:
\begin{enumerate}
\item Clear: a cut in a graph with at most 3 edges has size at most 3.
\item For any forest, each edge belongs to a random cut independently of any other edge. Hence, $q_F(F)=2^{-|F|}$
for any forest $F$. (See section \ref{sectionCutStatistics} for more details). Also, $q_\Box(F)=0$ and $q_F(F')=0$ for any two
distinct 4-forests $F$,$F'$.
\item Let the vertices of $K_4^-$ be labelled by $a,b,c,d$, where $a$ and $c$ are the vertices of degree 3.
Then a random cut in $K_4^-$ is isomorphic to $C_4$ if and only if $a$ and $c$ belong to one side of the cut,
and $b$ and $d$ to the other side. This happens with probability $1/8$ . Clearly, $q_F(K_4^-) =0$ for any 4-forest \(F\): \(K_{4}^{-}\) contains no 4-forest.
\item Finally, $|\lambda^{(2)}(G)|$ is the difference between
    two probabilities, hence is at most 1.
\end{enumerate}
\end{proof}
Taking a linear combination of the two OCC spectra from the previous claims gives us the
desired OCC spectrum, which completes the proof of our main theorem,
Theorem \ref{main}, when $p = 1/2$.
\begin{cor} \label{cor:OCCmain}
Let $\Lambda = \Lambda^{(1)} + \frac{16}{17}\gamma'\Lambda^{(2)}$. Then $\Lambda$ is
an OCC spectrum as described in Corollary \ref{corMain}:
\begin{itemize}
\item $\lambda_\emptyset =1$.
\item $\lambda_G = -1/7$ for all non-empty subgraphs $G$ of
    $K_3$ (and for the graph consisting of two disjoint edges).
\item Letting $\gamma = \frac{1}{17} \gamma' $ gives that
 $ \lambda_G \ge -1/7 + \gamma$ for any $G$ with more than three edges.
\end{itemize}
\end{cor}
\begin{proof}
Note that for any 4-forest $F$, the new eigenvalue $\lambda_F$ is now equal to $-1/7 + \frac{16}{17}\frac{1}{16}\gamma'$,
the eigenvalue $\lambda_{K_4^-}$ has increased to $-1/7 +  \frac{16}{17}\frac{1}{8}\gamma'$, and for all other non-empty graphs $G$ we have
$\lambda_G \ge -1/7 + \gamma' - \frac{16}{17}\gamma'$.
\end{proof}

\section{Cut Statistics}\label{sectionCutStatistics}
The purpose of this section is to study the cut statistics of graphs for a (uniform) random cut, in order to prove
Claim \ref{goodOCC1}. We begin by using block-decompositions of graphs to simplify our calculations.

We will sometimes think of a random cut in \(G\) as being produced by a random red/blue colouring of \(V(G)\), where each vertex is independently coloured red or blue with probability 1/2. For a red/blue colouring \(c\colon V(G) \to \{\text{red},\text{blue}\}\), we let \(Y(c)\) denote the number of edges in the associated cut, i.e. the number of multicoloured edges.

Let \(\mathcal{Q}(G) = (q_{k})_{k \geq 0}\) denote the distribution of \(|G \cap B|\); we call this the \emph{cut distribution} of \(G\). Let
\[Q_{G}(X) = \sum_{k\geq 0} q_{k}(G)X^{k}\]
denote the probability-generating function of \(|G \cap B|\). For example, if \(G\) is a single edge then \(q_{0}(-) = q_{1}(-) = 1/2\), and therefore
\[Q_{-}(X) = \tfrac{1}{2} + \tfrac{1}{2}X.\]
We will see that \(|G \cap B|\) is a sum of independent random variables \(|H \cap B|\), where \(H\) ranges over certain subgraphs of \(G\). Probability-generating functions will be a convenient tool for us, since if \(Y_1\) and \(Y_2\) are independent random variables, we have \(Q_{Y_1+Y_2}(X) = Q_{Y_1}(X)Q_{Y_{2}}(X)\).

\ynote{Here is some discussion on the origin of $\Lambda^{(1)}$.}

In the rest of the section, we will study the cut distribution in enough detail so that we can prove Claim~\ref{goodOCC1}. But first, let us digress and explain how to construct $\Lambda^{(1)}$. We begin by considering some small graphs and their cut distributions:
\[
\begin{array}{c|*{5}c}
 G & q_0(G) & q_1(G) & q_2(G) & q_3(G) & q_4(G) \\\hline
\emptyset & 1 & 0 & 0 & 0 & 0 \\
- & 1/2 & 1/2 & 0 & 0 & 0  \\
\wedge & 1/4 & 1/2 & 1/4 & 0 & 0 \\
\triangle & 1/4 & 0 & 3/4 & 0 & 0  \\
F_4 & 1/16 & 4/16 & 6/16 & 4/16 & 1/16  \\
K_4^- & 1/8 & 0 & 1/4 & 1/2 & 1/8
\end{array}
\]
In the table, $F_4$ is a forest with $4$ edges (they all have the same cut distribution).

Suppose we are looking for an OCC spectrum of the form
\[ \lambda(G) = (-1)^{|G|} \left[ c_0 q_0(G) + c_1 q_1(G) + c_2 q_2(G) + c_3 q_3(G) + c_4 q_4(G) \right]. \]
Since $\lambda(\emptyset) = 1$, $c_0 = 1$. Applying the proof of Theorem~\ref{ThmHoffman}
to a $\trium$, whose Fourier transform is concentrated  on subgraphs of a triangle,
shows that we need $\lambda(G) = \lambda_{\min} = -1/7$ for all subgraphs of the triangle.
This forces the choices $c_1 = -5/7$ and $c_2 = -1/7$.
Substituting $c_0,c_1,c_2$ into the equations defined by $F_4$ and $K_4^-$
gives us a lower and upper bound (respectively) on $4c_3 + c_4$.
Both bounds coincide (what luck! This good fortune does not hold for $p > 1/2$),
implying that $4c_3 + c_4 = 3/7$. To simplify matters, we choose $c_4 = 0$ and so $c_3 = 3/28$.

The OCC spectrum of $\Lambda^{(1)}$ is engineered to work for the graphs appearing in the table. In the rest of this section, we show that it also works for all other graphs.

\ynote{Here endeth the discussion.}

Observe that if \(G = G_1 \sqcup G_2\) then
\[Q_{G}(X) = Q_{G_1}(X)Q_{G_2}(X),\]
since \(G_1 \cap B\) and \(G_2 \cap B\) are independent, and \(|G \cap B| = |G_1 \cap B|+|G_2 \cap B|\).

Let \(G\) be a connected graph, and suppose that \(v\) is a \emph{cutvertex} of \(G\), meaning a vertex whose removal disconnects \(G\). Suppose the removal of \(v\) separates \(G\) into components \(G[S_1],\ldots,G[S_N]\). For each \(i\), let
\[H_{i} = G[S_{i} \cup \{v\}].\]
Observe that the system of random variables \(\{H_i \cap B:\ i \in [N]\}\) is independent, since for {\em any} vertex \(v\), the distribution of \(H \cap B\) remains unchanged even if we fix the class of the vertex \(v\), in which case the independence is immediate. Clearly,
\[|G \cap B| = \sum_{i=1}^{N}|H_{i} \cap B|.\]
It follows that
\[Q_{G}(X) = \prod_{i=1}^{N}Q_{H_i}(X).\]
Let \(H = \bigsqcup_{i} H_i\); \(H\) is produced by splitting the graph \(G\) at the vertex \(v\). (For example, splitting the graph \(\bowtie\) at the cutvertex in its centre produces the graph \(\rhd\ \lhd\).) Then
\[Q_{G}(X)=Q_{H}(X).\]

Recall that a {\em bridge} of a graph \(G\) is an edge whose removal increases the number of connected components of \(G\); a {\em block} of \(G\) is a bridge or a biconnected component of \(G\). Note that if \(G\) is bridgeless then \(q_1(G) = 0\), since a cut of size 1 would be a bridge.

Observe that if \(G\) and \(G'\) have the same number of bridges and the same number of blocks isomorphic to \(K\) for each biconnected graph \(K\), then \(G\) and \(G'\) have the same cut-distribution. In fact, if \(G\) has \(m\) bridges and \(t_{K}\) blocks isomorphic to \(K\) (for each biconnected graph \(K\)), then repeating the above splitting process within every component until there are no more cutvertices, we end up producing a graph \(G_{s}\) which is a vertex-disjoint union of all the blocks of \(G\). We call \(G_{s}\) the \emph{split} of \(G\). We have:
\[Q_{G}(X) = Q_{G_{s}}(X) = (\tfrac{1}{2} + \tfrac{1}{2}X)^{m} \prod_{K \in \mathcal{K}} (Q_{K}(X))^{t_{K}} = \tfrac{1}{2^{m}} (1+X)^{m}\prod_{K \in \mathcal{K}} (Q_{K}(X))^{t_{K}},\]
where \(\mathcal{K}\) denotes a set of representatives for the isomorphism classes of biconnected graphs. For example,
\[Q_{\elong} = (\tfrac{1}{2} + \tfrac{1}{2}X)(Q_{\lhd}(X))^{2} = (\tfrac{1}{2} + \tfrac{1}{2}X)(\tfrac{1}{4}+\tfrac{3}{4}X^{2})^{2}.\]

Now suppose \(G\) has exactly \(m\) bridges. Let \(H\) be the union of the biconnected components of \(G_{s}\); write
\[Q_{H}(X) = \sum_{i \geq 0}a_{i}X^{i}.\]
Here \((a_{i})_{i \geq 0}\) is the cut distribution of \(H\), so obviously, \(\sum_{i \geq 0}a_{i} = 1\). Note that \(a_{1}=0\), since \(H\) is bridgeless. We have
\begin{eqnarray}
\label{eq:pgfexp}
Q_{G}(X) & = & (\tfrac{1}{2}+\tfrac{1}{2}X)^{m}Q_{H}(X) \nonumber \\
& = & \tfrac{1}{2^{m}} (1+X)^{m}(a_{0}+a_{2}X^{2}+a_{3}X^3+\ldots) \nonumber\\
& = & \tfrac{1}{2^{m}}\left(1+mX+\tbinom{m}{2}X^2+\tbinom{m}{3}X^3+\ldots\right) \left(a_{0}+a_{2}X^{2}+a_{3}X^3+\ldots\right) \nonumber\\
& = & \tfrac{1}{2^{m}}\left(a_{0}+ma_{0}X+\left(\tbinom{m}{2}a_{0}+a_{2}\right)X^2+\left(\tbinom{m}{3}a_0+ma_{2}+a_{3}\right)X^3+R(X)X^4\right),
\end{eqnarray}
where \(R(X) \in \mathbb{Q}[X]\).

\subsection{Proof of Claim \ref{goodOCC1}} \label{sec:goodOCC1}
We will need the following additional facts about the cut distributions of graphs:
\begin{lem}\label{lemCutStatistics} Let \(G\) be a graph.
\begin{enumerate}
\item If \(G\) has exactly \(N\) connected components, then \(q_0(G) = 2^{N-v(G)}\).
\item If \(G\) has exactly \(m\) bridges, then \(q_1(G)=mq_0(G)\).
\item If \(G\) has a vertex with odd degree, then \(q_{k}(G) \leq 1/2\) for any \(k \geq 0\).
\item For any odd \(k\), \(q_{k}(G) \leq 1/2\).
\item Always \(q_{2}(G) \leq 3/4\).
\end{enumerate}
\end{lem}

\begin{proof} We follow the items of the lemma:
 \begin{enumerate}
\item If \(G\) has $N$ connected components then \(G \cap B = 0\) iff all the vertices of each connected component are given the same colour; the probability of this is \(2^{N-v(G)}\).
\item This follows immediately from equation (\ref{eq:pgfexp}).
 \item Let \(G\) be a graph with a vertex \(v\) of odd degree. For any red/blue colouring \(c\colon V(G) \to \{\textrm{red},\textrm{blue}\}\) of \(V(G)\), changing the colour of \(v\) produces a new colouring \(c'\) with \(Y(c') \neq Y(c)\). Since \((c')' = c\), \(c'\) determines \(c\). Denote by $Y_v(c),Y_v(c')$ the number of edges incident to $v$ which are cut in $c,c'$, respectively. Then $Y_v(c) + Y_v(c') = \deg(v)$, hence $Y_v(c) \neq Y_v(c')$; since $Y(c)-Y_v(c) = Y(c')-Y_v(c')$, necessarily $Y(c) \neq Y(c')$. Thus at most one cut of each pair $(c,c')$ cuts exactly $k$ edges.
\item By item 3, we may assume that all the degrees of \(G\) are even. Since a graph is Eulerian if and only if it is connected and all its degrees are even, every connected component of \(G\) is Eulerian. It follows that every cut in \(G\) has even size, and therefore \(q_{k}(G) = 0\).
\item The average number of edges in a random cut is $|G|/2$, and therefore
\[ |G|/2 = \sum_k kq_k(G) < 2 q_2(G) + (1-q_2(G))|G| = |G| + (2-|G|)q_2(G);\]
the inequality is strict because $q_0(G) > 0$.
Hence,
\[ q_2(G) < \frac{|G|/2}{2(|G|-2)} = \frac{1}{2} + \frac{1}{|G|-2}. \]
Therefore $q_2(G) < 3/4$ if $|G| \geq 6$. Assume from now on that $|G| \leq 5$.

Let $G_s$ be the split graph obtained by splitting $G$ into its blocks, as described above. If $G$ has any bridges, then $q_2(G) = q_2(G_s) \leq 1/2$, by 3. Otherwise, since each block has at least 3 edges and \(|G| \leq 5\), there is just one block, i.e. \(G = G_s\) is biconnected. Therefore $G$ is either a triangle, a \(C_4\), a \(C_5\) or a $K_4^-$. One may check that $q_2(K_3) = q_2(C_4) = 3/4$, $q_2(C_5) = 5/8$ and $q_2(K_4^-) = 1/4$.
 \end{enumerate}
\end{proof}

The following lemma encapsulates some trivial properties of graphs:
\begin{lem} \label{lem:trivialgraphics}
Let $G$ be a graph, and $H$ be the union of its biconnected components.
\begin{enumerate}
 \item We have $q_0(\emptyset) = 1$, $q_0(-)=1/2$, and $q_0(G)\leq 1/4$ for all other graphs.
 \item If $m = 0$ and $|G|$ is odd, then either $q_0(G) \leq 1/16$, or $G$ is a triangle or a $K_4^-$.
 \item Either $H = \emptyset$, or $a_0 \leq 1/4$.
\end{enumerate}
\end{lem}
\begin{proof}
 \begin{enumerate}
  \item Follows from Lemma~\ref{lemCutStatistics}(1).
  \item Since $m = 0$, every connected component of $G$ is biconnected, and so consists of at least three vertices. If $G$ has at least two connected components, then Lemma~\ref{lemCutStatistics}(1) implies that $q_0(G) \leq 1/16$, so we may assume that $G$ is connected. If $G$ has at least $5$ vertices, then again, $q_0(G) \leq 1/16$. The only remaining graphs are the triangle and $K_4^-$.
  \item The graph $H$ is a union of biconnected graphs. In particular, $H \neq -$. The item now follows from item~1.
 \end{enumerate}
\end{proof}

We can now prove Claim \ref{goodOCC1}.
\begin{proof}[of Claim \ref{goodOCC1}]
Write
\[f(G) = q_0(G)-\tfrac{5}{7} q_1(G) - \tfrac{1}{7}q_2(G) + \tfrac{3}{28}q_3(G).\]
The proof breaks into two parts: odd $|G|$ and even $|G|$.

\noindent \textbf{Proof for graphs with an odd number of edges:}
We will show that if \(|G|\) is odd then \(f(G) \leq \tfrac{1}{7}\), with equality if and only if \(G\) is an edge, a triangle, or \(K_{4}^{-}\), and that in all other cases, \(f(G) \leq \tfrac{1}{7}-\tfrac{1}{56}\).

By Lemma \ref{lemCutStatistics}, if \(G\) has exactly \(m\) bridges then \(q_{1} = mq_{0}\), so
\begin{equation}
 \label{eq:coeffs}
f(G) = (1-\tfrac{5}{7}m)q_{0}(G) - \tfrac{1}{7}q_2(G) + \tfrac{3}{28}q_3(G).
\end{equation}
First suppose \(m = 1\). In that case,
\[f(G) = \tfrac{2}{7}q_{0}(G) - \tfrac{1}{7}q_2(G) + \tfrac{3}{28}q_3(G).\]
If $G=-$ then $f(G) = -\frac{1}{7}$. Otherwise, Lemma~\ref{lem:trivialgraphics}(1) shows that \(q_{0}(G) \leq \tfrac{1}{4}\). By Lemma \ref{lemCutStatistics}(4), \(q_{3}(G) \leq \tfrac{1}{2}\), and therefore
\[f(G) \leq \tfrac{2}{7}\tfrac{1}{4} + \tfrac{3}{28}\tfrac{1}{2} = \tfrac{1}{8} = \tfrac{1}{7} - \tfrac{1}{56}.\]
If \(m \geq 2\), the coefficient of \(q_{0}(G)\) in equation~(\ref{eq:coeffs}) is negative, and therefore
\[f(G) < \tfrac{3}{28} = \tfrac{1}{7}-\tfrac{1}{28} < \tfrac{1}{7}-\tfrac{1}{56}.\]

From now on, we assume that \(m=0\). If $q_0(G) \leq \tfrac{1}{16}$, then using \(q_{3}(G) \leq \tfrac{1}{2}\), we obtain
\[f(G) \leq \tfrac{1}{16}+\tfrac{3}{28}\tfrac{1}{2} = \tfrac{13}{112} = \tfrac{1}{7}-\tfrac{3}{112} < \tfrac{1}{7}-\tfrac{1}{56},\]
so we are done. Otherwise, Lemma~\ref{lem:trivialgraphics}(2) implies that $G$ is either a triangle or $K_4^-$. One calculates explicitly that \(f(K_3) = f(K_4^-) = \frac{1}{7}\), completing the proof for all graphs with \(|G|\) odd.

\medskip

\noindent \textbf{Proof for graphs with an even number of edges:} We will show that if \(|G|\) is even then \(f(G) \geq -\tfrac{1}{7}\), with equality if and only if \(G\) is a 2-forest or a 4-forest, and that in all other cases, \(f(G) \geq -\tfrac{1}{7} + \tfrac{1}{28}\).

By equation~(\ref{eq:pgfexp}) we have:
\begin{eqnarray*}
f(G) & = & \tfrac{1}{2^{m}}\left[a_{0}-\tfrac{5}{7}ma_{0}-\tfrac{1}{7}\left(\tbinom{m}{2}a_{0}+a_{2}\right)+\tfrac{3}{28}\left(\tbinom{m}{3}a_0+ma_{2}+a_{3}\right)\right]\\
&  = & \tfrac{1}{2^{m}}\left[\left(1-\tfrac{5}{7}m-\tfrac{1}{7}\tbinom{m}{2}+\tfrac{3}{28} \tbinom{m}{3}\right)a_{0}+(-\tfrac{1}{7}+\tfrac{3}{28}m)a_{2}+\tfrac{3}{28}a_{3}\right].
\end{eqnarray*}
When \(m = 0\), i.e. every component of \(G\) is bridgeless,
\[f(G) = a_{0}-\tfrac{1}{7}a_{2}+\tfrac{3}{28}a_{3}.\]
By Lemma \ref{lemCutStatistics}(5), \(a_{2} \leq 3/4\), and therefore
\[f(G) > -\tfrac{1}{7}+\tfrac{1}{28}.\]
When \(m = 1\),
\[f(G) = \tfrac{1}{2}(\tfrac{2}{7}a_{0}-\tfrac{1}{28}a_{2}+\tfrac{3}{28}a_{3}) = \tfrac{1}{7}a_{0}-\tfrac{1}{56}a_{2}+\tfrac{3}{28}a_{3} > -\tfrac{3}{4} \tfrac{1}{56} = -\tfrac{1}{7}+\tfrac{29}{224} > -\tfrac{1}{7}+\tfrac{1}{28}.\]
When \(m=2\),
\[f(G) = \tfrac{1}{4}(-\tfrac{4}{7}a_{0}+\tfrac{1}{14}a_{2}+\tfrac{3}{28}a_{3}) = -\tfrac{1}{7}a_{0}+\tfrac{1}{56}a_{2}+\tfrac{3}{112}a_3.\]
We have \(f(G) = -\tfrac{1}{7}\) if and only if \(H = \emptyset\), i.e. \(G\) has exactly two edges. If \(H \neq \emptyset\), Lemma~\ref{lem:trivialgraphics}(3) implies that \(a_{0} \leq \tfrac{1}{4}\), and therefore
\[f(G) \geq -\tfrac{1}{28} = -\tfrac{1}{7}+\tfrac{3}{28} > -\tfrac{1}{7}+\tfrac{1}{28}.\]
When \(m = 3\),
\[f(G) = \tfrac{1}{8}(-\tfrac{41}{28}a_{0}+\tfrac{5}{28}a_{2}+\tfrac{3}{28}a_{3}) = -\tfrac{41}{224}a_{0}+\tfrac{5}{224}a_{2}+\tfrac{3}{224}a_{3}.\]
Since \(|G|\) is even, \(H \neq \emptyset\), so as above, \(a_{0} \leq \tfrac{1}{4}\). It follows that
\[f(G) \geq -\tfrac{41}{896} = -\tfrac{1}{7}+\tfrac{87}{896} > -\tfrac{1}{7}+\tfrac{1}{28}.\]
When \(m = 4\),
\[f(G) = \tfrac{1}{16}(-\tfrac{16}{7}a_{0}+\tfrac{2}{7}a_{2}+\tfrac{3}{28}a_{3}) = -\tfrac{1}{7}a_{0}+\tfrac{1}{56}a_{2}+\tfrac{3}{448}a_{3}.\]
We have \(f(G) = -\tfrac{1}{7}\) if and only if \(H = \emptyset\), i.e. \(G\) is a forest with 4 edges. Otherwise, \(a_{0} \leq \tfrac{1}{4}\), and therefore
\[f(G) \geq -\tfrac{1}{28} = -\tfrac{1}{7}+\tfrac{3}{28} > -\tfrac{1}{7}+\tfrac{1}{28}.\]

Finally, assume that \(m \geq 5\). Since the coefficients of \(a_{2}\) and \(a_{3}\) in \(f(G)\) are positive for \(m \geq 2\), we need only bound the coefficient of \(a_{0}\) away from \(-\tfrac{1}{7}\). Write
\[r(m) = \tfrac{1}{2^{m}}\left(1-\tfrac{5}{7}m-\tfrac{1}{7}\tbinom{m}{2}+\tfrac{3}{28} \tbinom{m}{3}\right)\]
for this coefficient. For \(m = 5\) we have
\[r(5) = -\tfrac{41}{448}.\]
Since \(e(G)\) is even, \(H \neq \emptyset\), and therefore \(a_{0} \leq \tfrac{1}{4}\), so
\[f(G) \geq -\tfrac{41}{448}\tfrac{1}{4} = -\tfrac{1}{7}+\tfrac{215}{1792} > -\tfrac{1}{7}+\tfrac{1}{28}.\]
For \(m = 6\), we have
\[r(6) = -\tfrac{23}{448},\]
and therefore
\[f(G) \geq -\tfrac{23}{448} = -\tfrac{1}{7}+\tfrac{41}{448}> -\tfrac{1}{7}+\tfrac{1}{28}.\]
For \(m = 7\), we have
\[r(7) = -\tfrac{13}{512}.\]
For \(m \geq 7\), the polynomial
\[1-\tfrac{5}{7}m-\tfrac{1}{7}\tbinom{m}{2}+\tfrac{3}{28} \tbinom{m}{3}\]
in the numerator of \(r\) is strictly increasing, and therefore
\[r(m) \geq -\tfrac{13}{512}\ \forall m \geq 7.\]
Hence,
\[f(G) \geq -\tfrac{13}{512} = -\tfrac{1}{7}+\tfrac{421}{3584}> -\tfrac{1}{7}+\tfrac{1}{28}\]
whenever \(m \geq 7\), completing the proof of Claim \ref{goodOCC1}.
\end{proof}

\section{\texorpdfstring{$p < 1/2$}{p < 1/2}}\label{secSmallp}
In this section, we explain how our method can be used to prove Theorem \ref{main} for all \(p \in (0,1/2)\). Note that when \(p < 1/2\), the intersecting and agreeing questions are no longer equivalent. Indeed, the triangle-agreeing family \(\mathcal{F}\) of all graphs containing no edges of a fixed triangle has \(\mu_{p}(\mathcal{F}) = (1-p)^3 > p^3\). For \(p < 1/2\), we will only be concerned with odd-cycle-intersecting families. 
\subsection{Skew analysis}
The general setting for skew Fourier analysis is the `weighted cube', i.e. \(\{0,1\}^{X}\) (where \(X\) is a finite set), endowed with the product measure
\[\mu_{p}(S) = p^{|S|}(1-p)^{|X|-|S|}\quad (S \subset X).\]
In our case, \(X = [n]^{(2)}\), the edge-set of the complete graph, so our probability space is simply $G(n,p)$. If \(G \subset K_n\), we define $\mu_p(G)$ to be the probability that $G(n,p)=G$, i.e.
$$\mu_{p}(G) = p^{|G|} (1-p)^{\binom{n}{2}-|G|},$$
and if $\cF$ is a family of graphs, we define \(\mu_{p}(\mathcal{F})\) to be the probability that \(G(n,p) \in \mathcal{F}\), i.e.
$$\mu_p(\cF) = \sum_{G \in \cF} \mu_p(G).$$
The measure \(\mu_p\) induces the following inner product on the vector space \(\mathbb{R}[\{0,1\}^{X}]\) of real-valued functions on \(\{0,1\}^{X}\):
$$
\ip{f,g} = \ip{f,g}_{p} = \Exp_{S \sim \mu_{p}}{[f(S)\cdot g(S)]} = \sum_{S \subset X} \mu(S) f(S)g(S) = \sum_{S \subset X}p^{|S|} (1-p)^{|X|-|S|} f(S)g(S).
$$
We define the \(p\)-skewed Fourier-Walsh basis as follows. For any \(e \in X\), let
$$
\chi_e(S) = \begin{cases}
      \sqrt{\frac{p}{1-p}} & \textrm{if }e \not \in S, \\
      -\sqrt{\frac{1-p}{p}} &  \textrm{if }e \in S. \end{cases}
$$
For each \(R \subset X\), let $\chi_R = \prod_{e \in R} \chi_e$. It is easy to see that \(\{\chi_{R}:\ R \subset X\}\) is an orthonormal basis for \((\mathbb{R}[\{0,1\}^{X}],\ip{,})\); we call it the (\(p\)-skewed) Fourier-Walsh basis. Every
$f : \{0,1\}^{X} \to \mathbb{R}$ has a unique expansion of the form
$$f = \sum_{R \subset X} \widehat{f}(R)\chi_R;$$
we have $\widehat{f}(R) = \ip{f,\chi_R}$ for each \(R \subset X\). We may call this the (\(p\)-skewed) Fourier expansion of \(f\). All the other formulas in section \ref{subsec:fourieranalysis} hold in the skewed setting also.

\begin{dfn}
For \(p < 1/2\), we define an OCC operator to be a linear operator $A$ on \(\mathbb{R}[\{0,1\}^{[n]^{(2)}}]\) such that
\begin{enumerate}
 \item If $f$ is the indicator-function of an odd-cycle-\emph{intersecting} family, then
\[f(G)=1 \Rightarrow Af(G)=0;\]
\item The Fourier-Walsh basis is a complete set of eigenfunctions of \(A\).
\end{enumerate}
\end{dfn}
(Note the change from odd-cycle-{\em agreeing} in the uniform-measure case.) As before, the set of OCC operators is a linear space.
\remove{It will be easier to describe the operators we will construct in matrix form,
 so we will introduce coordinates to our system.
To this end, we enumerate the edges of $K_n$ arbitrarily
from $0$ to $\binom{n}{2}-1$. We associate every subgraph of $K_n$ with its
characteristic vector of length ${\binom{n}{2}}$.
For any function $f : \mathbb{Z}_2^{\binom{n}{2}} \ra \mathbb{R}$ we represent $f$
by a column vector of length $2^{\binom{n}{2}}$.
The coordinate $f_i$ is $f(G)$, where $G$ is the graph whose characterstic vector
is the binary expansion of $i$.}

We will now construct a collection of OCC operators, one for each bipartite graph \(B\). Let
\[ M = \begin{pmatrix} \frac{1-2p}{1-p} & \frac{p}{1-p} \\ 1 & 0 \end{pmatrix};\]
we index the rows and columns of \(M\) with $\{0,1\}$.

Let \(B \subset K_n\) be a bipartite graph. For each edge $e$ of $K_n$, we define a \(2 \times 2\) matrix \(M_{B}^{(e)}\) as follows:
\[ M_{ B}^{(e)} = \begin{cases}  M &\text{if } e \in \barB;
\\  I_{2\times 2} &\text{if } e \in B, \end{cases} \]
where $I_{2\times 2}$ denotes the \(2 \times 2\) identity matrix. Finally, we define
\[ M_{B} = \bigotimes_{e \in K_n} M_{B}^{(e)}.\]
So $M_{B}$ is obtained from $M^{\otimes [n]^{(2)}}$
 by replacing $M$ with $I_{2 \times 2}$ for each edge of $B$; its rows and columns are indexed by \(\{0,1\}^{[n]^{(2)}}\). More explicitly, for any $G, H \subset K_n$,
$$
(M_{B})_{G,H} = \prod_{e \in K_n} (M_{B}^{(e)})_{G(e),H(e)}
$$
(where, of course, \(G(.)\) means the characteristic function of \(G\)). The matrix \(M\) was chosen so that
\begin{enumerate}
 \item $M_{1,1}=0$;
\item The skew Fourier-Walsh basis vectors
\begin{equation}
\label{eq:evecs}
\chi_{\emptyset} = \begin{pmatrix} 1 \\ 1 \end{pmatrix},\quad \chi_{\{e\}} = \begin{pmatrix} \sqrt{\frac{p}{1-p}} \\ -\sqrt{\frac{1-p}{p}}\end{pmatrix}
\end{equation}
are eigenvectors of \(M\).
\end{enumerate}
Note that these conditions determine $M$ uniquely up to multiplication by a scalar matrix. Together with the tensor product structure of \(M_{B}\), they guarantee that \(M_{B}\) has the respective properties of an OCC operator:
\begin{claim}
If $B$ is a bipartite graph, then the matrix $ M_{B}$ represents an OCC
operator when acting on functions by multiplying their vector representation from the left, i.e. by
\[(M_{B}f)(G) = \sum_{H \subset K_n} (M_{B})_{G,H}f(H).\]
For any graph $G \subset K_n$, the function $\chi_G$ is an eigenvector of \(M_{B}\) with eigenvalue
$$
\lambda_G = \left( -\frac{p}{1-p} \right)^{|G \cap \barB|} =
\left( -\frac{p}{1-p} \right)^{|G|} \left( -\frac{1-p}{p} \right)^{|G \cap B|}.
$$
\end{claim}
\begin{proof}
We need to show that if $\cF$ is odd-cycle-intersecting, then $\ip{f,M_{B}f} = 0$.
By linearity, it suffices to prove that for any $G,H \subset K_n$ with $G \cap H \cap \barB \neq \emptyset$, we have $(M_{B})_{G,H}=0$. Note that
$$
 (M_{B})_{G,H} = \prod_{e \in K_n} (M_{B}^{(e)})_{G(e),H(e)}.
$$
There exists $e \in \barB$ such that $G(e)=H(e)=1$; the corresponding multiplicand will be $M_{1,1}=0$, so $(M_{B})_{F,G}=0$, as required.

Note that the vectors (\ref{eq:evecs}) are simultaneously eigenvectors of \(M\) and \(I_{2 \times 2}\); the corresponding eigenvalues are \(1,-p/(1-p)\) (for \(M\)) and \(1,1\) (for \(I_{2 \times 2}\)). It follows by simple tensorization that for any graph $G \subset K_n$, the function $\chi_G$ is an eigenvector of \(M_{B}\) with eigenvalue
$$
\lambda_G = \left( -\frac{p}{1-p} \right)^{|G \cap \barB|} =
\left( -\frac{p}{1-p} \right)^{|G|} \left( -\frac{1-p}{p} \right)^{|G \cap B|}.
$$
\end{proof}

Note that \(M_{B}\) is the \(p\)-skew analogue of the operator \(A_{B}\) in the uniform case; indeed, when \(p = 1/2\), we have \(M_{0,0} = 0\), and therefore \(M_{B} = A_{B}\).

It is rather instructive to spend a moment studying the transpose \(M_{B}^{\top}\) of \(M_{B}\). By exactly the same argument
as above, whenever $f$ is the indicator-function of an odd-cycle-intersecting family, we have $\ip{M_{B}^{\top}f,f}=0$, as well as $\ip{f, M_{B}f}=0$ (although note that for \(p < 1/2\), it does not in general hold that $\ip{f,M_{B}g}=\ip{M_{B}^{\top}f,g}$.) Despite the fact that
the right eigenvectors of $M_{B}^{\top}$ (which are the left eigenvectors of $M_{B}$) are not the Fourier-Walsh basis, it turns out that the operator represented by $M_{B}^{\top}$ has an elegant
interpretation. For any two graphs $G$ and $H$, we define $G \oplus_p H$ not as a graph, but as a random graph,
formed as follows. Begin with the graph $G$. For every edge in $H$, if it is present in $G$ remove it,
and if it is absent from $G$ add it, independently at random with probability $\frac{p}{1-p}$ (here, we rely on
$p \le 1/2$). When $p=1/2$, the operation $\oplus_p$ degenerates into $\oplus$. Note that, as in the case
of $\oplus$, the distribution of $G \cap (G \oplus_p \barB)$ is supported on graphs contained in $B$.
We may lift the operation $(\cdot \oplus_p \barB)$ to an operator $N_B$:
$$
N_B f(G) = \Exp{[f(G \oplus_p \barB)]}.
$$
This operator is precisely $M_{B}^{\top}$. It has several nice properties. First and foremost, it is clear that when $f$ is the indicator function of an odd-cycle-intersecting family and $B$ is bipartite,
$$
f(G) = 1 \Rightarrow N_B f(G)=0.
$$
Secondly, it is a Markov operator representing a random walk on subgraphs of $K_n$, with stationary measure
$G(n,p)$, and the property that no two consecutive steps can intersect in an odd cycle.

\subsection{Engineering the eigenvalues for \texorpdfstring{$p < 1/2$}{p < 1/2}}
In this subsection, we construct an OCC operator with the necessary
spectrum for \(p \in [1/4,1/2)\), thus (almost) completing the proof of Theorem \ref{main}. In fact, in order to show that the constant $c_p$ in the stability part of Theorem \ref{main} is bounded if $p$ is bounded away from $0$, we will need to do this for a slightly extended interval. \remove{\dnote{can we remove the rest of the paragraph? it's a bit confusing.} If we do not want this extra `uniformity', we don't need the counterpart of $\Lambda^{(2)}$: there is an extra slackness that we can use to remove the extra tight graphs from the beginning. However, $\Lambda^{(2)}$ will be useful to obtain this uniformity.}

For the rest of this section, we assume that $p \in [\tau,1/2)$, where $\tau = 0.248$. The proof breaks down for slightly smaller $p$: the required inequality is violated by $3$-forests. However, as will be shown in section \ref{sec:smallp}, for $p$ in any closed sub-interval of $(0,1/4)$, Theorem \ref{main} follows from~\cite{ehud}.

We start by generalizing Lemma \ref{lemOCCSpectrum} and Corollary \ref{cor:qB}:
\begin{lem}\label{lemOCCSpectrum:p}
Let $\B$ be a distribution over bipartite graphs, and for every $B \in \B$ let $f_B$ be a real-valued
function whose domain is the set of subgraphs of $B$. Then
$$
\lambda_G = \left(-\frac{p}{1-p}\right)^{|G|} \Exp[f_B(B \cap G)]
$$
describes an OCC spectrum, where the expectation is over a random choice of $B$ from $\B$.
\end{lem}
\begin{proof}
 Trivial generalization of the proof of Lemma~\ref{lemOCCSpectrum}.
\end{proof}

\begin{cor}
\label{cor:qB:p}
Let $(V_1,V_2)$ be a random bipartition of the vertices of $K_n$, where each vertex is chosen independently to
belong to each $V_i$ with probability 1/2. Let $B$ be the set of edges of \(K_n\) between $V_1$ and $V_2$. For any graph $G \subseteq K_n$, let
$$
q_i(G) = \Pr[ |G \cap B|= i ],
$$
and for any bipartite graph $R$, let
$$
q_R(G) = \Pr[ (G \cap B) \isomorphicto R],
$$
where all probabilities are over the choice of the
random bipartition. Then for any integer $i$,
$$
 \lambda_G = \left(-\frac{p}{1-p}\right)^{|G|} q_i(G)
$$
 is an OCC spectrum, and for any bipartite graph $R$,
$$
\left(-\frac{p}{1-p}\right)^{|G|} q_R(G)
$$
is an OCC spectrum.
\end{cor}

Replacing `agreeing' with `intersecting', we have the following skewed analogue of Theorem \ref{ThmHoffman}:

\begin{thm}\label{skewedHoffman}
Let $\Lambda=(\lambda_G)_{G \subset K_n}$ be an OCC spectrum with \(\lambda_{\emptyset} = 1\), with minimal value
$\lambda_{\min}$ such that \(-1 < \lambda_{\min} < 0\), and with spectral gap $\gamma>0$.
Set $\nu = \frac{-\lambda_{\min}}{1-\lambda_{\min}}$ (so $\lambda_{\min}= \frac{-\nu}{1-\nu}$).
 Then for any odd-cycle-intersecting family $\cF$ of subgraphs of $K_n$, the following holds:
\begin{itemize}
 \item Upper bound: $\mu(\cF) \leq \nu$.
 \item Uniqueness: If $\mu(\cF) = \nu$, then
     $\widehat{\cF}(G)\not = 0$ only for $G \in \Lambda_{\min} \cup \{\emptyset\}$.
 \item Stability: Let $w = \sum_{G \not \in \Lambda_{\min} \cup
     \{\emptyset\}}\widehat{\cF}^2(G)$. Then $w \leq
     \frac{\nu}{(1-\nu)\gamma}(\nu-\mu(\cF)) .$
\end{itemize}
 \end{thm}

Similarly, we have the following analogue of Corollary~\ref{corMain}:
\begin{cor}\label{corMain:p}
Let \(p \in (0,1)\). Suppose that there exists an OCC spectrum $\Lambda$ with eigenvalues $\lambda_{\emptyset}=1$,
$\lambda_{\min}=-p^3/(1-p^3)$ and spectral gap $\gamma >0$. Assume that all
graphs in $\Lambda_{\min}$ (the set of graphs $G$ for which
$\lambda_G=\lambda_{\min}$) have at most $3$ edges. Then if $\cF$ is
an odd-cycle-intersecting family of subgraphs of $K_n$, the following holds:
\begin{itemize}
\item Upper bound: $\mu_{p}(\cF) \leq p^3$.
 \item Uniqueness: If $\mu_{p}(\cF) = p^3$, then $\cF$ is a $\trium$.
 \item Stability: If $\mu_{p}(\cF) > p^3 - \epsilon$, then there exists a $\trium$ $\cT$ such that $\mu_{p}(\cF \Delta \cT) = O_p\left(\epsilon\right)$.
\end{itemize}
\end{cor}
\begin{proof}
 Follows from Theorem \ref{skewedHoffman} much as Corollary \ref{corMain} follows from Theorem \ref{ThmHoffman}, with a small twist. The twist involves the proof of the stability part. Using Kindler-Safra, we construct a family $\cG$ depending on a set $A$ of at most $T_0$ coordinates which is $O_p(\epsilon)$-close to $\cF$. We can conclude stability (with the same proof as in the original corollary) if we can show that for $\epsilon$ small enough, $\cG$ must be odd-cycle-intersecting; the required bound on $\epsilon$ should depend only on $p$.

 Suppose $\cG$ isn't odd-cycle-intersecting. So there exist two graphs $G_1,G_2 \in \cG$ supported on $A$ which aren't odd-cycle-intersecting. For $i=1,2$, let $\cF_i = \{ J \subseteq \overline{A} : G_i \cup J \in \cF \}$. Since $\cF$ is odd-cycle-intersecting, the families $\cF_1,\cF_2$ must be cross-intersecting: any graph in $\cF_1$ intersects any graph in $\cF_2$. Using the cross-intersecting variant of Hoffman's bound~\cite[Theorem 13]{EFP}, the method of~\cite{ehud} shows that $\sqrt{\mu_p(\cF_1)\mu_p(\cF_2)} \leq p$. Therefore (without loss of generality) $\mu_p(\cF_1) \leq p$. This implies that $\mu_p(\cF \Delta \cG) \geq \mu_p(G_1)(1-p) \geq p^{|A|} (1-p) \geq p^{T_0} (1-p)$, where $\mu_p(G_1)$ is taken with respect to the edge set $A$. If $\epsilon$ is small enough, this contradicts the assumption that $\cG$ is $O_p(\epsilon)$-close to $\cF$.
\end{proof}
Our goal in this subsection is to exhibit an OCC spectrum satisfying the conditions of Corollary \ref{corMain:p}, for \(p \in (0,1/2)\). In section~\ref{sectionCutStatistics}, we explained how to choose $c_0,c_1,c_2,c_3,c_4 \in \mathbb{R}$ so that the OCC spectrum
\[ \lambda_G = \left(-\frac{p}{1-p}\right)^{|G|} \left[ c_0 q_0(G) + c_1 q_1(G) + c_2 q_2(G) + c_3 q_3(G) + c_4 q_4(G) \right]\]
satisfied the requirements of Corollary~\ref{corMain}. For general $p \in (0,1]$, the same calculations give the following constraints:
\begin{align*}
 c_0 &= 1, \\
 c_1 &= \frac{p^2-p-1}{p^2+p+1}, \\
 c_2 &= \frac{p^2-3p+1}{p^2+p+1},
\end{align*}
\[\frac{5p^2-27p+45-16/p}{p^2+p+1} \leq 4c_3 + c_4 \leq \frac{5p^2-27p+45-32/p+8/p^2}{p^2+p+1}. \]
When $p = 1/2$, the two bounds on $4c_3 + c_4$ coincide. When $p > 1/2$, they contradict one another, so the method fails. When $p < 1/2$, there is a gap, and choosing any value inside the gap, we get a spectrum which is \emph{not} tight on either $4$-forests or $K_4^-$. As before, we choose $c_4 = 0$. A judicious choice of $c_3$ is:
\[ c_3 = \frac{5p^2-27p+45-28/p+6/p^2}{4(p^2+p+1)};\]
this choice guarantees that $c_3 > 0$ for all \(p \in (0,1/2]\).

We are now ready to state the main claim of this section:
\begin{claim}\label{goodOCC1:p}
Let $\Lambda^{(1)}$ be the OCC spectrum described by
$$
\lambda^{(1)}_G = \left(-\frac{p}{1-p}\right)^{|G|} \left[ q_0(G) + c_1 q_1(G) + c_2 q_2(G) + c_3 q_3(G)\right],
$$
where $c_1,c_2,c_3$ are given by
\begin{align*}
 c_1 &= \frac{p^2-p-1}{p^2+p+1}, \\
 c_2 &= \frac{p^2-3p+1}{p^2+p+1}, \\
 c_3 &= \frac{5p^2-27p+45-28/p+6/p^2}{4(p^2+p+1)}.
\end{align*}
Then there exists $\gamma' > 0$ not depending on $n$ or $p$ such that
\begin{itemize}
\item $\lambda^{(1)}_\emptyset = 1$.
\item $\lambda^{(1)}_{\min}=-p^3/(1-p^3)$.
\item $\Lambda^{(1)}_{\min}$ consists of the following graphs: a single edge, a path of length two, two disjoint edges, and a triangle.
\item For all $H \not \in \Lambda_{\min} \cup \mathcal{F}_4 \cup \{ K_4^- \}$, we have $\lambda^{(1)}_H \ge -p^3/(1-p^3) + \gamma'$, where $\mathcal{F}_4$ denotes the set of $4$-forests.
\end{itemize}
\end{claim}

Before proving Claim \ref{goodOCC1:p}, we show that it implies Theorem \ref{main}. We have the following analogue of Claim \ref{goodOCC2}:
\begin{claim}\label{goodOCC2:p}
Let $\Lambda^{(2)}$ be the OCC spectrum described by
$$
\lambda^{(2)}_G = \left(-\frac{p}{1-p}\right)^{|G|} \left[\sum_{F \in \mathcal{F}_{4}} q_F(G) - q_{\Box}(G)\right],
$$
where $\Box$ denotes $C_4$.
Then \begin{enumerate}
\item $\lambda^{(2)}_H = 0$ for all $H$ with less than 4 edges.
\item $\lambda^{(2)}_F= 2^{-4} p^4/(1-p)^4$ for all 4-forests \(F\).
\item $\lambda^{(2)}_{K_4^-}=2^{-3} p^5/(1-p)^5$.
\item $|\lambda^{(2)}_G| \le 1$ for all $G$.
\end{enumerate}
\end{claim}
\begin{proof}
 Same as the proof of Claim \ref{goodOCC2}, using the fact that $|p/(1-p)| \leq 1$ to prove the last item.
\end{proof}
We have the following analogue of Corollary \ref{cor:OCCmain}:

\begin{cor} \label{cor:OCCmain:p}
Let $\Lambda = \Lambda_1 + \frac{16}{17}\gamma'\Lambda_2$. Then $\Lambda$ is
an OCC spectrum as described in Corollary \ref{corMain:p}:
\begin{itemize}
\item $\lambda_\emptyset =1$.
\item $\lambda_G = -p^3/(1-p^3)$ for all non-empty subgraphs $G$ of
    $K_3$ (and for the graph consisting of two disjoint edges).
\item Letting $\gamma = \frac{p^4}{17(1-p)^4} \gamma' \geq \frac{\tau^4}{17(1-\tau)^4}$ gives
 $ \lambda_G \ge -\frac{p^3}{1-p^3} + \gamma$ whenever $|G| > 3$.
\end{itemize}
\end{cor}
\begin{proof}
Same as the proof of Corollary \ref{cor:OCCmain}, only $\lambda_F,\lambda_{K_4^-}$ are somewhat smaller. We use the fact that $p/(1-p) = 1/(1-p) - 1$ is an increasing function of $p$.
\end{proof}

This implies Theorem~\ref{main} for $p \in [\tau,1/2)$. The rest of the proof is found in subsection~\ref{sec:smallp}.

\subsection{Proof of Claim \ref{goodOCC1:p}} \label{sec:goodOCC1:p}

The proof of Claim~\ref{goodOCC1:p} uses Lemmas~\ref{lemCutStatistics} and~\ref{lem:trivialgraphics}, and in principle follows the same route as the proof of Claim \ref{goodOCC1} for \(p=1/2\). However, whereas in the case of $p=1/2$ we could verify all the estimates with explicit calculations, here we need to argue that certain inequalities (which are fixed, i.e. do not depend on \(n\)) hold for the entire range $p \in [\tau,1/2)$. The inequalities in question will always be of the form $r(p) > \min\{r_i(p) : i \in S\}$, where $r,r_i$ are explicit rational functions. We actually verify the stronger claim that $r(p) > r_i(p)$ for all $i \in S$. Each such inequality is equivalent to an inequality $P_i(p) > 0$, for some polynomials $P_i$. These inequalities can be checked by verifying that $P_i(3/8) > 0$ (note that $\tau < 3/8$), and that $P_i(x)$ has no zeroes in $[\tau,1/2)$; the latter can be verified formally using Sturm chains (see for example \cite{HookMcAree}). This verification has been done for all 
inequalities of this form appearing below.

We will prove Claim \ref{goodOCC1:p} by reducing it to a finite number of cases (similarly to the proof of Claim~\ref{goodOCC1}), and showing that $\lambda_G > -p^3/(1-p^3)$ for all graphs not in $\Lambda_{\min}$. This automatically implies the existence of a spectral gap $\gamma_{p} > 0$, which might depend on $p$. If, however, we restrict ourselves to graphs other than $4$-forests and $K_4^{-}$, then all the inequalities are strict on $[\tau,1/2]$: one can verify that the corresponding polynomial \(P\) has \(P(3/8) >0\), and no zeros in $[\tau,1/2]$. So in these cases, by compactness, the minimum spectral gap on the entire interval $[\tau,1/2]$ is \(\geq \gamma'\) for some \(\gamma' >0\) not depending on \(p\).

We will need some easy facts about graphs in addition to Lemma~\ref{lem:trivialgraphics}:
\begin{lem} \label{lem:trivialgraphics:p}
 Let $G$ be a graph with $m$ bridges.
\begin{enumerate}
 \item If $m = 1$ and $|G|>1$, then $|G| \geq 4$.
 \item If $m = 0$ and \(|G| \leq 5\), then $G$ is a triangle, a $C_4$, a $C_5$ or a $K_4^-$.
\end{enumerate}
\end{lem}
\begin{proof}
\begin{enumerate}
 \item Every biconnected graph has at least $3$ edges.
 \item If $G$ has two biconnected components, then $|G| \geq 6$. The only biconnected graphs with at most $5$ edges are those given in the list.
\end{enumerate}
\end{proof}

\begin{proof}[of Claim \ref{goodOCC1:p}]
We begin by noting that for $p \in (0,1/2]$, $c_0$ and $c_3$ are always positive, and $c_1$ is always negative. The remaining coefficient $c_2$ changes signs
from positive to negative at $(3-\sqrt{5})/2 = 0.382$ (to 3 d.p.). Knowing the signs of the coefficients will help us estimate $\lambda_G$.

The rest of the proof consists of two parts: $|G|$ odd and $|G|$ even.

\noindent \textbf{Proof for graphs with an odd number of edges:}
Lemma~\ref{lemCutStatistics}(4,5) implies the general bound
\[ \lambda_G \geq -\left(\frac{p}{1-p}\right)^{|G|}
 \left[ q_0(1 + mc_1) + \max\left(\tfrac{3}{4}c_2,0\right) + \tfrac{1}{2}c_3 \right]. \]
It can be checked that $1+c_1 > 0$, whereas $1+mc_1 < 0$ for $m \geq 2$.

When $m \geq 2$, since $1+mc_1 < 0$, we have the sharper estimate
\[ \lambda_G \geq -\left(\frac{p}{1-p}\right)^{|G|}
 \left[ \max\left(\tfrac{3}{4}c_2,0\right) + \tfrac{1}{2}c_3 \right]. \]
If $|G|=3$ then $G$ is a $3$-forest, and we can verify that $\lambda_G > -p^3/(1-p^3)$ by direct calculation. Otherwise, $-(p/(1-p))^{|G|} \geq -(p/(1-p))^5$, so
\[ \lambda_G \geq -\left(\frac{p}{1-p}\right)^5
 \left[ \max\left(\tfrac{3}{4}c_2,0\right) + \tfrac{1}{2}c_3 \right]. \]
It can be checked that the right-hand side is always $> -p^3/(1-p^3)$.

When $m=1$, Lemma~\ref{lem:trivialgraphics:p}(1) implies that either $|G|=1$ or $|G| \geq 5$. In the former case, $\lambda_G = -p^3/(1-p^3)$. In the latter case, Lemma~\ref{lem:trivialgraphics}(1) implies that $q_0 \leq 1/4$, and therefore
\[ \lambda_G \geq -\left(\frac{p}{1-p}\right)^5
 \left[ \tfrac{1}{4}(1 + c_1) + \max\left(\tfrac{3}{4}c_2,0\right) + \tfrac{1}{2}c_3 \right]. \]
It can be checked that the right-hand side is always $> -p^3/(1-p^3)$.

When $m=0$, Lemma~\ref{lem:trivialgraphics:p}(2) shows that either $G$ is a triangle, $C_5$ or $K_4^-$, or $|G|\geq 7$. If $G$ is a triangle then $\lambda_G = -p^3/(1-p^3)$. If $G$ is $C_5$ or $K_4^-$, we can verify that $\lambda_G > -p^3/(1-p^3)$ by direct calculation, except that for $K_4^-$, we get equality when $p=1/2$. Otherwise, Lemma~\ref{lem:trivialgraphics}(2) shows that $q_0 \leq 1/16$, and so
\[ \lambda_G \geq -\left(\frac{p}{1-p}\right)^7
 \left[ \tfrac{1}{16} + \max\left(\tfrac{3}{4}c_2,0\right) + \tfrac{1}{2}c_3 \right]. \]
It can be checked that the right-hand side is always $> -p^3/(1-p^3)$.

\medskip

\noindent \textbf{Proof for graphs with an even number of edges:}
Equation~\eqref{eq:pgfexp} implies that
\[ \lambda_G = \left(\frac{p}{1-p}\right)^{|G|} (d_0(m) a_0 + d_2(m) a_2 + d_3(m) a_3), \]
where $d_0,d_2,d_3$ are defined by
\begin{align*}
 d_0(m) &= 2^{-m} \left[1 + mc_1 + \binom{m}{2}c_2 + \binom{m}{3}c_3\right], \\
 d_2(m) &= 2^{-m} (c_2 + mc_3), \\
 d_3(m) &= 2^{-m} c_3.
\end{align*}
Since $c_3 > 0$, we know that $d_3(m) > 0$. We can further check that $d_2(m) > 0$ when $m \geq 2$; this just involves checking that $c_2 + 2c_3 > 0$.

We claim that $d_1(m) > 0$ for $m \geq 10$. To see this, check first that $c_1 + 7c_3 > 0$ and $c_2 + 2c_3 > 0$. Note that
\begin{align*}
2^{m+1} d_0(m+1) - 2^m d_0(m) &= c_1 + mc_2 + \binom{m}{2} c_3 \\
&\geq (c_1 + 7c_3) + m(c_2 + 2c_3) > 0,
\end{align*}
using $\binom{m}{2} \geq 2m+7$, which is true for $m \geq 7$. It remains to check by direct calculation that $d_1(10) > 0$.

We have shown that when $m \geq 10$, $\lambda_G > 0$. If $m < 10$ and $G$ is a forest, then $G$ is either a $2$-forest, a $4$-forest, a $6$-forest or an $8$-forest. If $G$ is a $2$-forest, then $\lambda_G = -p^3/(1-p^3)$. For the other forests listed, direct calculation shows that $\lambda_G > -p^3/(1-p)^3$, except that for $4$-forests, we get equality when $p=1/2$.

The remaining case is when $m < 10$ and $G$ is \emph{not} a forest. Lemmas~\ref{lemCutStatistics}(5) and~\ref{lem:trivialgraphics}(3) give the following bound:
\[ \lambda_G \geq \left(\frac{p}{1-p}\right)^2
\left[ \min\left(\tfrac{1}{4}d_0(m),0\right) + \min\left(\tfrac{3}{4}d_2(m),0\right)\right].\]
It can be checked that for all $m < 10$, the right-hand side is $> -p^3/(1-p^3)$.
\end{proof}

\subsection{Small \texorpdfstring{$p$}{p}} \label{sec:smallp}
In this section, we complete the proof of Theorem~\ref{main} by considering the range $p \in (0,\tau]$. We read off the OCC spectrum constructed in~\cite{ehud} and analyze it. For the rest of the section, we assume that $p \in (0,\tau]$.

\begin{claim} \label{OCC:smallp}
 The following describes an OCC spectrum $\Lambda$:
\[ \lambda_G = \lambda_{|G|} = \left(-\frac{p}{1-p}\right)^{|G|} \left[ 1 - \frac{1+p}{1+p+p^2}|G| + \frac{1}{1+p+p^2} \binom{|G|}{2} \right].\]
Moreover,
\begin{enumerate}
\item $\lambda_0 =1$;
\item $\lambda_{\textrm{min}} = \lambda_{1} = \lambda_{2} = \lambda_{3} = -p^3/(1-p^3)$; \(\Lambda_{\textrm{min}}\) consists of all graphs with 1,2 or 3 edges.
\item \[\lambda_{|G|} \geq \lambda_{5} = -\left(\frac{p}{1-p}\right)^5\left(\frac{6-4p+p^2}{1+p+p^2}\right)\]
whenever \(|G| \geq 4\), so the spectral gap
\[\gamma = \frac{p^3}{1-p^3}-\left(\frac{p}{1-p}\right)^5\left(\frac{6-4p+p^2}{1+p+p^2}\right).\]
\end{enumerate}
\end{claim}
\begin{proof}
This can be deduced from \cite{ehud}. Alternatively, Lemma \ref{lemOCCSpectrum:p} implies that any function of the form
\[\lambda_{G} = \lambda_{|G|} = \left(-\frac{p}{1-p}\right)^{|G|}\left(a_{0}+a_{1} |G| + a_2 \binom{|G|}{2}\right)\quad (a_0,a_1,a_2 \in \mathbb{R})\]
is an OCC spectrum: \(\binom{|G|}{i}\) simply counts the number of \(i\)-edge subgraphs of \(G\), and graphs with 1 or 2 edges are bipartite. The coefficients chosen above are forced by \(\lambda_0 = 1\), \(\lambda_1 = \lambda_2 = -p^3/(1-p^3)\); it is easily checked that the above choice also guarantees that \(\lambda_3 = -p^3/(1-p^3)\). For the rest, one may calculate that:
\begin{itemize}
\item \(\lambda_{5} < 0\);
\item \(\lambda_{|G|} \geq 0\) whenever \(|G| \geq 4\) is even;
\item \(|\lambda_{|G|+2}| < |\lambda_{|G|}|\) whenever \(|G| \geq 5\) is odd,
\end{itemize}
completing the proof.
\end{proof}

To deduce Theorem~\ref{main} from Corollary~\ref{corMain:p}, we require only the following easy lemma:
\begin{lem} \label{lem:main:finish}
 Let $\gamma$ be the spectral gap in Claim~\ref{OCC:smallp}. Then
\[\frac{p^3}{(1-p^3)\gamma}\]
is bounded from above for $p \in (0,\tau]$.
\end{lem}
\begin{proof}
Let
\[g(p) = \frac{(1-p^3)\gamma}{p^3} = 1-\frac{p^2(6-4p+p^2)}{(1-p)^4}\]
It is easy to check that for \(p \in [0,1/4]\), \(g(p)\) is a strictly decreasing function of \(p\), with \(g(0)=1\) and \(g(1/4) = 0\). It follows that \(g(p) \geq g(\tau) > 0\) for all \(p \in [0,\tau]\). Hence,
\[\frac{p^3}{(1-p^3)\gamma} \leq \frac{1}{g(\tau)}\]
for all $p \in (0,\tau]$, as required.
\end{proof}

Lemma \ref{lem:main:finish} and Corollary \ref{cor:OCCmain:p} imply that there exists an absolute constant $C$ such that if $\cF$ is an odd-cycle-intersecting family with $\mu_p(\cF) \geq p^3 - \epsilon$, then
\[ \sum_{|G|>3} \hat{\cF}^2(G) \leq C\epsilon.\]
We now appeal to Theorem 3 in Kindler-Safra \cite{kindler-safra}, which in fact is stated for the \(p\)-skew measure. (Note that we quote inferior bounds, for brevity.)
\begin{thm}[Kindler-Safra]
 For every $t \in \mathbb{N}$ and $p \in (0,1)$, there exist positive reals $\epsilon_{0} = \Omega(p^{4t})$, $c = O(p^{-t})$
and $T = O(tp^{-4t})$ such that the following holds. Let \(N \in \mathbb{N}\), and let $f:\{0,1\}^{N} \to \{0,1\}$ be a Boolean function such that
 $$
 \sum_{|S|>t} \widehat{f}(S)^2 = \epsilon < \epsilon_0.
 $$
 Then there exists a Boolean function $g:\{0,1\}^{N} \to \{0,1\}$, depending on at most $T_0$ coordinates, such that
 $$
 \mu_p(\{R: f(R)\not = g(R)\} ) \leq c \varepsilon.
 $$
\end{thm}
Note that if \(p \in [\delta,1/2)\), where \(\delta >0\) is fixed, then \(\epsilon_0,c\) and \(T_0\) can be chosen to depend only upon \(\delta\). By the same argument as in the proof of Corollary \ref{corMain}, it follows that $\cF$ is $(c_p C \epsilon)$-close to a $\trium$, where \(c_p\) depends only upon \(p\), and is bounded for \(p \in [\delta,1/2)\) for any fixed \(\delta >0\), completing the proof of Theorem~\ref{main}.

\section{Odd-linear-dependency-intersecting families of subsets of \texorpdfstring{$\{0,1\}^n$}{\{0,1\}\textasciicircum n}}\label{sectionSchur}
In this section we prove Theorem~\ref{thmSchurIntro}.

\subsection{Definitions and Results} \label{sectionSchurIntro}
As stated in the Introduction, we say that a family \(\mathcal{F}\) of hypergraphs on \([n]\) is {\em odd-linear-dependency-intersecting} (or {\em odd-LD-intersecting}, for short) if for any \(G,H \in \mathcal{F}\) there exist \(l \in \mathbb{N}\) and nonempty sets \(A_1,A_2,\ldots,A_{2l+1} \in G \cap H\) such that
\[A_1 \Delta A_2 \Delta \ldots \Delta A_{2l+1} = \emptyset.\]

If we identify subsets of \([n]\) with their characteristic vectors in \(\{0,1\}^n = \mathbb{Z}_{2}^{n}\), then the symmetric difference operation \(\Delta\) is identified with vector-space addition, and hypergraphs on \([n]\) are identified with subsets of \(\{0,1\}^n\). Hence, equivalently, we say that a family \(\mathcal{F}\) of subsets of \(\mathbb{Z}_{2}^{n}\) is {\em odd-LD-intersecting} if for any two subsets \(S,T \in \mathcal{F}\) there exist \(l \in \mathbb{N}\) and non-zero vectors \(v_1,v_2,\ldots,v_{2l+1} \in S \cap T\) such that
\[v_1 + v_2 + \ldots + v_{2l+1} = 0.\]
In other words, the intersection of the two subsets must contain a non-trivial odd linear dependency.

Similarly, we say that a family \(\mathcal{F}\) of subsets of \(\mathbb{Z}_{2}^{n}\) is {\em odd-LD-agreeing} if for any \(S,T \in \mathcal{F}\) there exist \(l \in \mathbb{N}\) and non-zero vectors \(v_1,v_2,\ldots,v_{2l+1} \in \overline{S \Delta T}\) such that
\[v_1 + v_2 + \ldots + v_{2l+1} = 0.\]

Since \(0\) cannot occur in a non-trivial odd linear dependency, it is irrelevant: if \(\mathcal{F}\) is a maximal odd-LD agreeing family of subsets of \(\mathbb{Z}_{2}^{n}\), then $S \cup \{0\} \in \cF$ iff $S \setminus \{0\} \in \cF$.
Hence, from now on we will consider only families of subsets of \(\mathbb{Z}_{2}^{n}\) not containing \(0\), i.e. families of hypergraphs not containing \(\emptyset\) as an edge. Therefore we will work in \(\{0,1\}^{n} \setminus \{0\} = \mathbb{Z}_{2}^{n} \setminus \{0\}\). This will make our proofs neater, since the 0-vector behaves differently from all other vectors in \(\mathbb{Z}_{2}^{n}\).

For \(p \in [0,1]\), the skew product measure \(\mu_{p}\) on \(\mathbb{Z}_{2}^{n} \setminus \{0\}\) is defined, naturally, as follows. For \(S \subset \mathbb{Z}_{2}^{n} \setminus \{0\}\), we define
\[\mu_{p}(S) = p^{|S|}(1-p)^{2^n-1-|S|},\]
i.e. the probability that a \(p\)-random subset of \(\mathbb{Z}_{2}^{n} \setminus \{0\}\) is equal to \(S\), and if \(\mathcal{F}\) is a family of subsets of \(\mathbb{Z}_{2}^{n} \setminus \{0\}\), we define
\[\mu_{p}(\mathcal{F}) = \sum_{S \in \mathcal{F}} \mu_{p}(S).\]
We will work mostly with the uniform measure \(\mu_{1/2}\), which we will write as \(\mu\).

A {\em Schur triple} of vectors in \(\mathbb{Z}_{2}^{n} \setminus \{0\}\) is a set of three vectors \(\{x,y,z\}\) such that \(x+y=z\) (i.e. \(x+y+z=0\)) --- equivalently, a linearly dependent set of size 3. We say that a family \(\mathcal{T}\) of subsets of \(\mathbb{Z}_{2}^n \setminus \{0\}\) is a {\em Schur junta} if there exists a Schur triple \(\{x,y,x+y\}\) such that \(\mathcal{T}\) consists of all subsets of \(\mathbb{Z}_{2}^{n} \setminus \{0\}\) with prescribed intersection with \(\{x,y,x+y\}\). Similarly, we say that \(\mathcal{T}\) is a {\em Schur-umvirate} if there exists a Schur triple \(\{x,y,x+y\}\) such that \(\mathcal{T}\) consists of all subsets of \(\mathbb{Z}_{2}^{n} \setminus \{0\}\) containing \(\{x,y,x+y\}\).

An odd linear dependency will be the analogue of an odd cycle. We have the following:
\begin{thm}\label{thmSchur}
If $\cF$ is an odd-LD-agreeing family of subsets of \(\mathbb{Z}_{2}^{n} \setminus \{0\}\) then
\[\mu(\mathcal{F}) \leq 1/8.\]
Equality holds if and only if \(\mathcal{F}\) is a Schur junta. Moreover, there exists a constant \(c\) such that for any \(\epsilon >0\), if \(\mathcal{F}\) is an odd-LD-agreeing family of subsets of \(\mathbb{Z}_{2}^{n} \setminus \{0\}\) with \(\mu(\mathcal{F}) > \tfrac{1}{8}-\epsilon\), then there exists a Schur junta \(\mathcal{T}\) such that
\[\mu(\mathcal{T} \Delta \mathcal{F}) \le c \epsilon.\]
\end{thm}

A similar result holds for the skew product measures:

\begin{thm}\label{skewSchur}
\begin{itemize}
\item{{\bf [Extremal Families]}}
Let \(p \leq 1/2\). If $\cF$ is an odd-LD-intersecting family of subsets of \(\mathbb{Z}_{2}^{n} \setminus \{0\}\), then
\[\mu_p(\cF) \le p^3.\]
Equality holds if and only if $\cF$ is a Schur-umvirate.
\item{{\bf [Stability]}} There exists a constant $c$ such that for any $\epsilon \ge 0$,
if $\cF$ is an odd-LD-intersecting
family with \(\mu_{p}(\F) \geq p^3 - \epsilon\) then there exists a Schur-umvirate $\T$ such that
$$
\mu_p ( \T \Delta \cF) \le c \epsilon.
$$
\end{itemize}
\end{thm}

We may deduce Theorem \ref{main} from this by `lifting' a family of graphs to a family of subsets of \(\mathbb{Z}_{2}^{n} \setminus \{0\}\), in the obvious way. In detail, let \(\mathcal{F}\) be an odd-cycle-intersecting family of graphs. Let
\[\mathcal{H} = \{F \cup S:\ F \in \mathcal{F}, S \subset \{0,1\}^{n} \setminus ([n]^{(2)} \cup \{0\})\};\]
then \(\mu_{p}(\mathcal{H}) = \mu_{p}(\mathcal{F})\), and \(\mathcal{H}\) is odd-LD-intersecting, since an odd cycle is lifted to an odd linear dependency.

It seems impossible to deduce Theorem \ref{thmSchur} from Theorem \ref{main}, so Theorem \ref{thmSchur} is in some sense a {\em bona-fide} generalization. The calculations required to prove Theorem \ref{thmSchur} require one extra special case to be checked, but are in some ways simpler and more elegant, suggesting that this is the correct setting for our ideas. Indeed, we make crucial use of the fact that the {\em ground set} (as well as its power set) lives inside a vector space over \(\mathbb{Z}_{2}\).

We will focus on the case of the uniform measure, Theorem \ref{thmSchur}, and only mention briefly how to prove Theorem \ref{skewSchur}.

\subsection{Cayley operators} \label{sectionSchurCayley}
First, some preliminaries. If \(S \subset \mathbb{Z}_2^n \setminus \{0\}\), we write \(\rank(S) = \dim(\Span(S))\) for the dimension of the subspace spanned by \(S\). Let \(I(S) = \{v \in S:\ v \notin \Span(S \setminus \{v\})\}\) be the subset of \(S\) consisting of vectors which do not appear in any linear dependency of \(S\), and let \(m(S) = |I(S)|\). We write \(J(S) = S \setminus I(S)\) for the union of the linearly dependent subsets of \(S\).

For \(x,y \in \mathbb{Z}_{2}^n\), we write
\[\langle x,y \rangle = \sum_{i=1}^{n}x_i y_i\]
for the standard bilinear form on \(\mathbb{Z}_{2}^{n}\). If \(S \subset \mathbb{Z}_{2}^n\), we write
\[S^{\perp} = \{x \in \mathbb{Z}_{2}^{n}:\ \langle x,v \rangle = 0\}.\]
Then \(S^{\perp}\) is a subspace of \(\mathbb{Z}_{2}^{n}\), satisfying \(\dim(S^{\perp})+\dim(\Span(S)) = n\).

Recall that an {\em affine subspace} of a vector space \(V\) is a subset of \(V\) of the form \(U+a\), where \(U\) is a vector subspace of \(V\), and \(a \in V\) --- i.e., it is a translate of a subspace. If \(U\) has dimension \(d\), then \(U+a\) is also said to have dimension \(d\). If \(V\) has dimension \(n\), an affine subspace of \(V\) with dimension \(n-1\) is called an {\em affine hyperplane} of \(V\).

We will need the following easy lemma:
\begin{lem}
\label{lemma:affinesubspace}
If \(S = \{v_1,\ldots,v_d\} \subset \mathbb{Z}_{2}^{n}\) is linearly independent, then for any \(r_1,\ldots,r_d \in \{0,1\}\), the set
\[A:=\{x \in \mathbb{Z}_2^n:\ \langle v_i,x \rangle = r_i\ \forall i \in [d]\}\]
is a translate of \(S^{\perp}\), and is therefore an affine subspace with dimension \(n-d\).
\end{lem}
\begin{proof}
For each \(i \in [d]\) choose a vector \(y_i \in (S \setminus \{v_i\})^{\perp} \setminus S^{\perp}\). Note that \((S\setminus \{v_{i}\})^{\perp}\) is a subspace of dimension \(n-d+1\), and \(S^{\perp}\) is a subspace of dimension \(n-d\), so \((S \setminus \{v_i\})^{\perp} \setminus S^{\perp}\) is certainly nonempty. Moreover, \(\langle y_i,v_j \rangle = \delta_{i,j}\). Let
\[a = \sum_{i=1}^{d} r_i y_i;\]
then
\[A = a+S^{\perp},\]
as required.
\end{proof}

In particular, if \(w \in \mathbb{Z}_{2}^{n} \setminus \{0\}\) then a set of the form
\[A_w := \{v \in \mathbb{Z}_2^n:\ \langle v,w \rangle = 1\}\]
is an affine hyperplane. The following simple observation drives our whole approach:

\begin{lem}
An affine hyperplane of the form
\[A_w = \{v \in \mathbb{Z}_2^n:\ \langle v,w \rangle = 1\}\quad (w \in \mathbb{Z}_{2}^{n} \setminus \{0\})\]
contains no odd linear dependency.
\end{lem}
\begin{proof}
If \(v_1,\ldots,v_{2l+1} \in A_w\) then
\[\langle \sum_{i=1}^{2l+1} v_i,w \rangle = \sum_{i=1}^{2l+1} \langle v_i,w \rangle = 1,\]
so \(\sum_{i=1}^{2l+1}v_i \neq 0\).
\end{proof}

This motivates the following
\begin{dfn}
A {\em hyperplane subset} is a subset of some \(A_w\), where \(w \in \mathbb{Z}_{2}^{n}\).
\end{dfn}

Observe that if \(A\) is a hyperplane subset, \(\mathcal{F}\) is odd-LD-agreeing, and \(S \in \mathcal{F}\), then \(S \oplus \overline{A} \notin \mathcal{F}\). A hyperplane subset will be the analogue of a bipartite graph. Indeed, a bipartite graph \(G\) with bipartition \((X,Y)\) is a subset of \(A_w\) if we define \(w_i = \boldsymbol{1}\{i \in X\}\).

If \(w\) is a vector chosen uniformly at random from \(\mathbb{Z}_2^n\), and \(S \subset \{0,1\}^n\), we write \(Y_{S} = |S \cap A_w|\). Note that \(S \cap A_w\) is the analogue of a random cut in a graph. Indeed, if \(S \subset [n]^{(2)}\), i.e. \(S\) is a graph, then \(S \cap A_w\) is precisely a random cut in \(S\), and \(Y_S\) is the number of edges in a random cut in \(S\). We write \(Q_{S}(X)\) for the probability-generating function of the random variable \(Y_S\).

We now have a Cayley graph on \(\mathbb{Z}_{2}^{\{0,1\}^n \setminus \{0\}}\) (rather than \(\mathbb{Z}_{2}^{[n]^{(2)}}\)), where the generating set is \(\{\overline{A}: A\textrm{ is a hyperplane subset}\}\).

\begin{dfn}\label{OLDC}
A linear operator $A$ on real-valued functions on $\mathbb{Z}_2^{\{0,1\}^n \setminus \{0\}}$ will be called {\em Odd-Linear-Dependency-Cayley},
or {\em OLDC} for short, if it has the following two properties:
\begin{enumerate}
\item If $\cF$ is an odd-LD-agreeing family, and $f$ is its characteristic function, then
\[f(S)=1\Rightarrow Af(S)=0.\]
\item The Fourier-Walsh basis is a complete set of eigenfunctions of \(A\).
\end{enumerate}
The vector of eigenvalues of an OLDC operator, indexed by the subsets of $\{0,1\}^{n} \setminus \{0\}$, will be called an {\em
OLDC spectrum}.
\end{dfn}
We have the following analogue of Corollary \ref{cor:qB}:

\begin{cor}
\label{cor:genqB}
Let $w$ be a uniform random vector in \(\mathbb{Z}_2^n\), i.e. where each component is independently chosen to be \(0\) or \(1\) with probability \(1/2\). Let
\[A_w = \{v \in \mathbb{Z}_2^n:\ \langle v,w \rangle = 1\}.\]
For any subset \(S \subset \mathbb{Z}_{2}^n\), let
$$
q_i(S) = \Pr[ |S \cap A_w|= i ],
$$
and for any set \(R \subset \mathbb{Z}_{2}^n\) with no non-trivial odd linear dependency, let
$$
q_R(S) = \Pr[ (S \cap A_w) \isomorphicto R],
$$
where $H \isomorphicto R$ means that there is a linear isomorphism of \(\mathbb{Z}_2^n\) mapping $H$ to $R$, and all probabilities are over the choice of the
random vector \(w\). Then for any integer $i$,
$$
 \lambda_S = (-1)^{|S|} q_i(S)
$$
 is an OLDC spectrum, and for any set \(R \subset \mathbb{Z}_{2}^n\) with no non-trivial odd linear dependency,
$$
(-1)^{|S|} q_R(S)
$$
is an OLDC spectrum.
\end{cor}

\subsection{Construction of the OLDC spectrum} \label{sectionSchurCutStatistics}
Note that if \(S = \{v_1,\ldots,v_d\}\) is linearly independent, then \(\{\langle v_i,w \rangle:\ i \in [d]\}\) is an independent system of \(\textrm{Bin}(1,1/2)\) random variables. To see this, observe that for any \(r_1,\ldots,r_d \in \{0,1\}\), by Lemma \ref{lemma:affinesubspace},
\[\{w \in \mathbb{Z}_2^n:\ \langle v_i,w \rangle = r_i\ \forall i \in [d]\}\]
is an affine subspace of \(\mathbb{Z}_{2}^{n}\) of dimension \(n-d\), and therefore has size \(2^{n-d}\).

It follows that for each \(v \in I(S)\), the \(\textrm{Bin}(1,1/2)\) random variable \(\langle v,w \rangle\) is independent of the system \(\{\langle v',w \rangle:\ v' \in S \setminus \{v\}\}\). Hence, for any \(S\), if \(I(S) = \{v_1,\ldots,v_m\}\) and \(J(S) = J = \{u_1,\ldots,u_l\}\) then
\[Y_S = \sum_{i=1}^{m} Y_{\{v_i\}} + Y_{J}.\]
write
\[Q_{J}(X) = \sum_{i \geq 0}a_{i}X^{i};\]
note that \(a_{1}=0\). We have
\begin{eqnarray}
\label{eq:pgfexpansion}
Q_{S}(X) & = & (\tfrac{1}{2}+\tfrac{1}{2}X)^{m}Q_{J}(X) \nonumber \\
& = & \tfrac{1}{2^{m}} (1+X)^{m}(a_{0}+a_{2}X^{2}+a_{3}X^3+\ldots) \nonumber\\
& = & \tfrac{1}{2^{m}}\left(1+mX+\tbinom{m}{2}X^2+\tbinom{m}{3}X^3+\ldots\right) \left(a_{0}+a_{2}X^{2}+a_{3}X^3+\ldots\right) \nonumber\\
& = & \tfrac{1}{2^{m}}\left(a_{0}+ma_{0}X+\left(\tbinom{m}{2}a_{0}+a_{2}\right)X^2+\left(\tbinom{m}{3}a_0+ma_{2}+a_{3}\right)X^3+R(X)X^4\right),
\end{eqnarray}
where \(R(X) \in \mathbb{Q}[X]\); this is the exact analogue of (\ref{eq:pgfexp}).

The `same' spectrum which worked before turns out to work in the new setting also:

\begin{claim}\label{gengoodOCC1}
Let $\Lambda^{(1)}$ be the OLDC spectrum described by
$$
\lambda^{(1)}_S = (-1)^{|S|} \left[ q_0(S) -\frac{5}{7}q_1(S) -\frac{1}{7}q_2(S) + \frac{3}{28} q_3(S)\right],
$$
Then
\begin{itemize}
\item $\lambda^{(1)}_\emptyset = 1.$
\item $\lambda^{(1)}_{\min}=-1/7.$
\item $\Lambda^{(1)}_{\min}$ consists of all singletons, all sets of size 2, all linearly independent sets of size 4, and all sets of the form \(\{x,y,z,x+y,x+z\}\).
\item For all $S \not \in \Lambda^{(1)}_{\min}$ it holds that
    $\lambda^{(1)}_S \ge -1/7 + \gamma'$, with $\gamma' = 1/56$.
\end{itemize}
\end{claim}

We remark that sets of size 2 are the analogue of 2-forests, linearly independent sets of size 4 are the analogue of 4-forests, and \(\{x,y,z,x+y,x+z\}\) is the analogue of \(K_4^{-}\). All sets of size 2 (and all linearly independent sets of size 4) are isomorphic, so unlike in the graph case, we need not distinguish between them.

In order to prove Claim \ref{gengoodOCC1}, we need the following generalization of Lemma \ref{lemCutStatistics}:
\begin{lem}
\label{lem:genstatistics} Let \(S\) be a set of vectors.
\begin{enumerate}
\item \(q_0(S) = 2^{-\rank(S)}\).
\item \(q_1(S)=m(S)q_0(S) = m(S)2^{-\rank(S)}\).
\item If there exists \(i \in [n]\) such that \(|\{v \in S: v(i)=1\}|\) is odd, then \(q_{k}(S) \leq 1/2\) for any \(k \geq 0\).
\item For any odd \(k\), \(q_{k}(S) \leq 1/2\).
\item Always \(q_{2}(S) \leq 3/4\).
\end{enumerate}
\end{lem}

\begin{proof} We follow the items of the lemma.
\begin{enumerate}
\item Observe that \(q_0(S)\) is the probability that \(w\) lies in the subspace \(S^{\perp}\), which has dimension \(n - \rank(S)\), and therefore size \(2^{n-\rank(S)}\), proving 1.
\item Observe that \(|S \cap A_w|=1\) if and only if \(w \in (S \setminus \{v\})^{\perp} \setminus S^{\perp}\) for some \(v \in I(S)\): if $v \in S \cap A_w$ participates in some linear dependency $v = \sum v_i$, then linearity of the inner product implies that not all the $v_i$ can be outside $A_w$. The sets \(\{(S \setminus \{v\})^{\perp} \setminus S^{\perp}:\ v \in I(S)\}\) are disjoint, and each has size \(2^{n-\rank(S)}\), proving 2.
\item Let
\[T_k = \{w \in \{0,1\}^n:\ |S \cap A_w|=k\}.\]
Observe that for any \(w \in T_k\), \(w + e_i \notin T_k\), where \(e_i\) denotes the vector \((0,0,\ldots,0,1,0,\ldots,0)\) with a 1 in the \(i\)th place (cf. the corresponding part in the proof of Lemma~\ref{lemCutStatistics}). It follows that \(|T_k| \leq 2^{n-1}\), i.e. \(q_k(S) \leq 1/2\), proving 3.
\item By item 3, we may assume that for each \(i \in [n]\), \(|\{v \in S: v(i)=1\}|\) is even. But then for any \(w \in \mathbb{Z}_2^n\), \(\sum_{s \in S} \langle s,w \rangle = 0\) (since every $w(i)$ is summed an even number of times), and therefore \(|\{s \in S:\ \langle s,w \rangle = 1\}|\) is even. Hence \(q_k(S)=0\) for any odd \(k \in \mathbb{N}\), proving 4.
\item The average size of \(|S \cap A_w|\) is \(|S|/2\), and therefore
\[ |S|/2 = \sum_k kq_k(S) < 2 q_2(S) + (1-q_2(S))|S| = |S| + (2-|S|)q_2(S);\]
the inequality is strict because $q_0(S) > 0$. Hence,
\[ q_2(S) < \frac{|S|}{2(|S|-2)} = \frac{1}{2} + \frac{1}{|S|-2}.\]
Therefore $q_2(S) < 3/4$ if $|S| \geq 6$. Assume from now on that $|S| \leq 5$.

By item 3, we may assume that for each \(i \in [n]\), \(|\{v \in S: v(i)=1\}|\) is even, and therefore \(\sum_{v\in S}v = 0\). Let \(T \subset S\) be the smallest linearly dependent subset of \(S\). Since \(|S| \leq 5\), \(T\) must be of the form \(\{x,y,x+y\}\), \(\{x,y,z,x+y+z\}\), or \(\{x,y,z,v,x+y+z+v\}\). Since \(S\) sums to zero, \(S \setminus T\) must also sum to zero, but \(|S \setminus T| \leq 2\), and no set of size 1 or 2 sums to zero. Hence, \(S \setminus T = \emptyset\), i.e. \(S=T\). One may check that
\[q_2(\{x,y,x+y\}) = q_2(\{x,y,z,x+y+z\}) = 3/4,\quad q_2(\{x,y,z,v,x+y+z+v\}) = 5/8,\]
proving 5.
\end{enumerate}
\end{proof}

We also need a counterpart of Lemma \ref{lem:trivialgraphics}:
\begin{lem} \label{lem:trivialsetics}
 Let $S$ be a set of vectors.
\begin{enumerate}
 \item We have $q_0(\emptyset) = 1$, $q_0(\{x\}) = 1/2$, and $q_0(S) \leq 1/4$ for all other sets.
 \item If $m(S) = 0$ and $|S|$ is odd, then either $q_0(S) \leq 1/16$ or $S$ is of the form
\[ \{x,y,x+y\},\ \{x,y,z,x+y,x+z\},\ \textrm{or } \{x,y,z,x+y,y+z,x+z,x+y+z\}.\]
 \item Either $J(S) = \emptyset$ or $a_0 \leq 1/4$.
\end{enumerate}
\end{lem}
\begin{proof}
 \begin{enumerate}
  \item If $S \neq \emptyset,\{x\}$ then $\rank S \geq 2$, so the item follows from Lemma~\ref{lem:genstatistics}(1).
  \item If $\rank(S) \geq 4$ then item 1 implies that $q_0(S) \leq 1/16$. If \(\rank(S) \leq 3\), the only possibilities for \(S\) are those listed.
  \item If $J(S) \neq \emptyset$ then $\rank(J(S)) \geq 2$, since there is no linear dependency in $\{x\}$, so the item follows from item 1.
 \end{enumerate}
\end{proof}
We note that in item 2, the first possibility corresponds to a triangle, and the second to a $K_4^-$. The third possibility has no graph counterpart.

As before, these lemmas enable us to prove Claim \ref{gengoodOCC1}:

\begin{proof}[of Claim \ref{gengoodOCC1}]
The proof of Claim~\ref{goodOCC1} relied on equation (\ref{eq:pgfexp}) and Lemmas~\ref{lemCutStatistics} and~\ref{lem:trivialgraphics}. In order to prove Claim~\ref{gengoodOCC1}, we replace those by equation  (\ref{eq:pgfexpansion}) and Lemmas~\ref{lem:genstatistics} and~\ref{lem:trivialsetics}. The proof goes through line-by-line if we replace $G$ with $S$, $m$ with $m(S)$ and $H$ with $J(S)$. The proof uses the unconditional estimates of Lemma~\ref{lemCutStatistics} and the conditional estimates of Lemma~\ref{lem:trivialgraphics}; these carry through in the present setting. In a few places, $f(G)$ was explicitly calculated for some graphs; those are either forests, or the exceptional graphs of Lemma~\ref{lem:trivialgraphics}(2). In the present setting, we require exactly the same explicit calculations for the corresponding sets, and there is just one other exceptional structure to deal with,
\[S=\{x,y,z,x+y,x+z,y+z,x+y+z\},\]
which has
\[Q_{S}(X)=\frac{1}{8} + \frac{7}{8}X^4.\]
In this case, we explicitly calculate $f(S) = 1/8 = 1/7 - 1/56$. This completes the proof of Claim~\ref{gengoodOCC1}.

\end{proof}

We have the following analogue of Claim \ref{goodOCC2}:
\begin{claim}\label{goodOCC2gen}
Let $\Lambda^{(2)}$ be the OLDC spectrum described by
$$
\lambda^{(2)}_S = (-1)^{|S|} \left[\sum q_{I}(S) - q_{\Box}(G)\right]
$$
where \(I\) denotes a linearly independent set of size 4, and $\Box$ denotes the set \(\{a,b,c,a+b+c\}\).
Then \begin{enumerate}
\item $\lambda^{(2)}_S = 0$ for all $S$ with \(|S| \leq 3\).
\item $\lambda^{(2)}_S= 1/16$ if \(S\) is a linearly independent set of size 4.
\item $\lambda^{(2)}_{\{x,y,z,x+y,x+z\}}=1/8$.
\item $|\lambda^{(2)}_S| \le 1$ for all $S$.
\end{enumerate}
\end{claim}
\begin{proof}
We follow the items of the claim:
\begin{enumerate}
\item Clear.
\item For any linearly independent sets \(I,I'\) of size 4, \(q_{I}(I') = q_{4}(I') = 1/16\). Also $q_\Box(I')=0$.
\item Let \(S = \{x,y,z,x+y,x+z\}\). Clearly, $q_I(S) =0$ for any linearly independent set \(I\) of size 4. Note that the only subset of \(S\) isomorphic to \(\{a,b,c,a+b+c\}\) is \(\{y,z,x+z,x+z\}\), and therefore \(q_{\Box}(S) = 1/8\), proving 3.
\item Finally, $|\lambda^{(2)}(S)|$ is the difference between
    two probabilities, hence is at most 1.
\end{enumerate}
\end{proof}

Exactly the same argument as before now shows that if \(\mathcal{F}\) is an odd-LD-agreeing family of subsets of \(\{0,1\}^n \setminus \{0\}\), then \(\mu(\mathcal{F}) \leq 1/8\). If \(\mathcal{F}\) is odd-LD-intersecting, we may deduce from Lemma \ref{lem:ehud} that equality holds only if \(\mathcal{F}\) consists of all families of subsets containing a fixed Schur triple \(\{x,y,x+y\}\). If \(\mathcal{F}\) is an odd-LD-agreeing family of subsets of \(\{0,1\}^n \setminus \{0\}\), we may deduce using the same monotonization argument as in Lemma \ref{agree2intersect} that equality holds only if \(\mathcal{F}\) is a Schur junta. Stability follows by the same argument as before.

\subsection{\texorpdfstring{$p < 1/2$}{p < 1/2}} \label{sectionSchurSkew}
We now outline briefly how the skew-measure analogue, Theorem \ref{skewSchur}, can be proved using the technique of section~\ref{secSmallp}. This time the proof of the main claim, Claim~\ref{goodOCC1:p}, relies also on Lemma~\ref{lem:trivialgraphics:p}. It is easy to extend this lemma to the current setting:
\begin{lem} \label{lem:trivialsetics:p}
 Let $S$ be a set of vectors.
\begin{enumerate}
 \item If $m(S) = 1$ and $|S|>1$ then $|S| \geq 4$.
 \item If $m(S) = 0$ and $|S| \leq 5$, then \(S\) is of the form
\[ \{x,y,x+y\}, \{x,y,z,x+y+z\}, \{x,y,z,w,x+y+z+w\}, \textrm{or }\{x,y,z,x+y,x+z\}.\]
\end{enumerate}
\end{lem}
\begin{proof}
\begin{enumerate}
 \item The smallest linear dependency is $\{x,y,x+y\}$.
 \item Easy enumeration. Note that $\{x,y,z,x+y,x+y+z\}$ is isomorphic to the last member in the list.
\end{enumerate}
\end{proof}
All sets in item~(2) correspond to graphs: a triangle, $C_4$, $C_5$, and $K_4^-$, respectively. Therefore the only new case to check is the extra case in Lemma~\ref{lem:trivialsetics}(2). In this case also, we have $\lambda_G > -p^3/(1-p^3)$, and so Claim~\ref{goodOCC1:p} remains true in the current setting. This leads to a proof of Theorem \ref{skewSchur}.

\section{Discussion}\label{secDiscussion}
There are many intriguing generalizations of the problems discussed in this paper. We mention a few of them below,
and state several conjectures.
\subsection{Cross-triangle-intersecting families}
Many, or perhaps most, of the interesting theorems about intersecting families can be generalized to cross-intersecting families. We say that two families of graphs, $\cF$ and $\G$, are {\em cross-triangle-intersecting} if for
every $F \in \cF$ and $G \in \G$ the intersection $F \cap G$ contains a triangle. A natural conjecture is the following:
\begin{conj}
Let $\cF$ and $\G$ be cross-triangle-intersecting families of graphs on the same set of $n$ vertices.
Then $\mu(\cF)\mu(\G) \leq (1/8)^2$. Equality holds if and only if $\cF = \G$ is a $\trium$.
\end{conj}
The standard technique of extending spectral proofs {\em \`{a} la} Hoffman from the intersecting case to
the cross-intersecting case requires only one small additional piece of data. The minimal eigenvalue, $\lambda_{\min}$, must also be the second largest in absolute value. In our case, the OCC spectrum we have tailored does not have this property: the 3-forests have eigenvalue $41/224$, which is greater than $1/7$. It seems plausible that by using more of the \(q_{R}\)'s, one can construct an OCC spectrum with the required property, but we believe that the calculations required will be substantially more involved.
\subsection{\texorpdfstring{$p > 1/2$}{p > 1/2}}
As we explained, our preliminary construction of an OCC spectrum of the form
$$\lambda_G = (-1)^{|G|} \sum_{i \geq 0} c_i q_i(G) $$
enjoyed a certain amount of luck, since the upper and lower bounds that were imposed on  $4c_3+c_4$ by the 4-forests and by $K_4^-$ coincided. For $p > 1/2$, our luck runs out, as the bounds contradict each other, and a more sophisticated construction
is required. So far we have not been able to fix this flaw, but we see no theoretical barrier that rules out
a spectral proof of our main theorem for all $p \le 3/4$. Indeed, we conjecture that Theorem \ref{main} holds
for all $p \leq 3/4$. Easy homework for the reader: why does the theorem fail for $p >3/4$?
(Hint: Mantel's theorem.)
\subsection{Other intersecting families}
The definition of an odd-cycle-intersecting family of graphs is clearly a special case of the definition of a $\G$-intersecting family for any family of graphs $\G$.
\begin{dfn}
For a family of graphs $\G$, let
$$
m(\G) = \sup_n \{ \max \mu(\cF)\ :\ \cF \ \mbox{is a}\ \G\mbox{-intersecting family of graphs on } n\ \mbox{vertices}\}.
$$
For a fixed graph, $G$ we abbreviate $m(\{G\})$ to $m(G)$.
We also will refer to $m_p(\G)$ when the measure in question is the skew product measure with parameter $p$.
\end{dfn}
Here is a sample of known facts and questions concerning $m(\G)$ for various choices of $\G$.
\begin{itemize}
\item
It was observed by Noga Alon \cite{Alon} that for every star forest $G$, $m(G)=1/2$. He further
conjectured that there is an $\epsilon>0$ so that for every $G$ which is not a
star forest $m(G)<1/2-\epsilon$,  and pointed out that this
holds for all non-bipartite graphs $G$ and that it suffices
to prove the conjecture for $P_3$, the path with 3 edges.

 An intriguing fact is that the simplest guess, $m(P_3) = 1/8$ 
 (conjectured in \cite{cgfs}), is false. 
 Demetres Christofides \cite{christofides} has constructed a \(P_3\)-intersecting family of graphs on 6 vertices, 
 with measure \(17/128 > 1/8\).
\item The obvious conjecture generalizing our main theorem is that if $\G_k$ denotes the
family of non-$k$-colorable graphs, then $m_p(\G_k) = p^{\C{k+1}{2}}$ for all $p \le \frac{2k-1}{2k}$,
with equality only for $K_{k+1}$-umvirates. It is quite plausible that, at least for small values of $k$ and $p= 1/2$, this conjecture will be amenable
to our methods.
\item It seems to us, perhaps for lack of imagination, that the $\trium$ might be extremal not only for
odd-cycle-intersecting families, but also for the more general case of cycle-intersecting families.
If true, this would hold only for $p \le 1/2$. An indication that this may be a significantly harder question
is the fact that for $p=1/2$ there is a neck-to-neck race for maximality between the $\trium$ and
the family of all graphs with at least $\tfrac{1}{2} \C{n}{2}+ \frac{1}{2}n$ edges, and to settle the result one needs
to consult the table of the normal distribution. Moreover, the generalization of this statement to non-uniform hypergraphs is false. A cycle-intersecting family of graphs corresponds to a linear-dependency-intersecting family of subsets of \(\mathbb{Z}_{2}^{n} \setminus \{0\}\), but it is easy to construct such a family with measure \(1/2 - o(1)\). (Take all sets of vectors with cardinality at least \(2^{n-1}+(n+1)/2\). The intersection of any two is a set of at least \(n+1\) vectors, and is therefore linearly dependent. Standard estimates show that this family has measure \(1/2 - o(1)\).)
\end{itemize}

There are many other interesting structures (other than a graph structure) that one may impose on the ground set. 
An example studied in \cite{cgfs} is the cyclic group \(\mathbb{Z}_{n}\) of integers modulo \(n\).
 For \(B \subset \mathbb{Z}_{n}\), we say that a family \(\mathcal{F}\) of subsets of \(\mathbb{Z}_{n}\) is 
 \(B\){\em -translate-intersecting} if the intersection of any two sets in \(\mathcal{F}\) contains a 
 translate of \(B\). The authors conjecture that a \(B\)-translate-intersecting family of subsets of 
 \(\mathbb{Z}_{n}\) has size at most \(2^{n-|B|}\). They prove this in the case where \(B\) is an interval; 
 Paul Russell~\cite{russell} has given a different, algebraic proof. 
 F\"uredi, Griggs, Holzman and Kleitman~\cite{furedigriggs} have proved it in the case \(|B|=3\). 
 Griggs and Walker~\cite{Griggs198990} prove that for each $B$, 
 the conjecture holds for infinitely many values of $n$.
  For most configurations \(B\), the question (for all $n$) remains open.

\subsection{Connection to entropy?}
\label{subsec:entropyconnection}
The OCC spectrum that we constructed can be expressed in the form $\sum c_{\B} A_{\B}$, where the sum of the coefficients $c_{\B}$ is 1. However, this is not a convex combination, as some of the coefficients are negative.
We observe that if we restrict ourselves to non-negative coefficients, our proof method cannot give a bound better than 1/4.
At the risk of falling prey to mundane numerology, we cannot help but wonder if there is a connection to the
bound of 1/4 that one gets using entropy. We raise this question due to the other superficial resemblances
between our approach and that of \cite{cgfs}.
\vspace{5mm}\newline\noindent
{\bf \large Acknowledgements: } We would like to thank Vera S\'{o}s and Noga Alon for useful conversations.

\bibliographystyle{plain}
\bibliography{K3agree}
\end{document}